\documentclass[reqno]{amsart}
\newcommand \datum {January 24, 2023}

\usepackage{amssymb,latexsym}
\usepackage{amsmath}
\usepackage{graphicx}
\usepackage[dvipsnames]{xcolor}
\usepackage{enumerate}
\numberwithin{equation}{section}
\theoremstyle{plain}
 \newtheorem{theorem}{Theorem}[section]
  
 \newtheorem{lemma}[theorem]{Lemma}

 \newtheorem{observation}[theorem]{Observation} 
 \newtheorem{corollary}[theorem]{Corollary}
\theoremstyle{definition}
 \newtheorem{definition}[theorem]{Definition}
 \newtheorem{example}[theorem]{Example}
 
 \newtheorem{remark}[theorem]{Remark}
 
 \newtheorem{convention}[theorem]{Convention}
\theoremstyle{remark}
 \newtheorem{case}{Case}

\newenvironment{enumeratei}{\begin{enumerate}[\upshape (i)]}%
                            {\end{enumerate}}   

\newcommand \hi[1]{#1^\ast}
\newcommand \lo[1]{#1_\ast}
\newcommand \OT [1] {\textup{OT}(#1)}
\newcommand \LHOT [1] {\textup{LHOT}(#1)}
\newcommand \TopInt [1] {\textup{TopInt}(#1)}
\newcommand \lsupp [1] {\textup{lsupp}(#1)}
\newcommand \rsupp [1] {\textup{rsupp}(#1)}
\newcommand \LBnd [1] {\textup{LBnd}(#1)}
\newcommand \RBnd [1] {\textup{RBnd}(#1)}
\newcommand \Bnd [1] {\textup{Bnd}(#1)}
\newcommand \jsum[1] {\mathrel{+_{#1}}}
\newcommand \brosum[3] {#1_{#2\leftarrow\sharp #3}}

\newcommand \height[1] {\textup{ht}(#1)}

\newcommand \Shk [1]  {\mathsf S_7^{(#1)}}

\newcommand \lcorner [1] {\textup{lc}(#1)}
\newcommand \rcorner [1] {\textup{rc}(#1)}

\newcommand \con  {\textup{con}}
\newcommand \intv [1]{\mathfrak{#1}}
\newcommand \nn{\intv n}
\newcommand \pp{\intv p}
\newcommand \qq{\intv q}
\newcommand \rr{\intv r}
\newcommand \hhl {\intv h_{\textup{left}}}
\newcommand \hhr {\intv h_{\textup{right}}}

\newcommand \Floor [1] {\textup{Floor}(#1)}
\newcommand \LFloor [1] {\textup{LFloor}(#1)}
\newcommand \RFloor [1] {\textup{RFloor}(#1)}
\newcommand \Roof [1] {\textup{Roof}(#1)}
\newcommand \LRoof [1] {\textup{LRoof}(#1)}
\newcommand \RRoof [1] {\textup{RRoof}(#1)}

\newcommand \FullRect [1] {\textup{FullRect}(#1)}
\renewcommand \phi{\varphi}

\newcommand \RR {\mathbb R}
\newcommand\spiralfil{curl}
\newcommand \Nwl[1] {\textup{Nwl}(#1)}
\newcommand \Nel[1] {\textup{Nel}(#1)}
\newcommand \Foot [1] {\textup{Foot}(#1)}
\newcommand \Peak [1] {\textup{Peak}(#1)}
\newcommand \CircR [1] {\textup{CircR}(#1)}
\newcommand \UHCircR [1] {\textup{UHCircR}(#1)}

\newcommand \Jir [1] {\textup{J}(#1)} 
\newcommand \Mir [1] {\textup{M}(#1)}

\newcommand \QQ {\mathbb Q}
\newcommand \Nnul {\mathbb N_0}
\newcommand \Nplu {\mathbb N^+}
\newcommand \Min [1] {\textup{Min}(#1)}

\newcommand \gideal {\mathord{\downarrow_{\textup{g}}}}
\newcommand \gfilter {\mathord{\uparrow_{\textup{g}}}}
\newcommand \Ofilt [1] {\mathcal F_{\textup{ord}}(#1)}
\newcommand \Lfilt [1] {\mathcal F_{\textup{lat}}(#1)}
\newcommand \Oid [1] {\mathcal I_{\textup{ord}}(#1)}
\newcommand \Lid [1] {\mathcal I_{\textup{lat}}(#1)}
\newcommand \ideal {\mathord\downarrow\kern0.5pt }
\newcommand \pideal[1] {\mathord\downarrow_{\kern-1.0pt #1}\kern0.5pt}
\newcommand \pfilter[1] {\mathord\uparrow_{\kern-1.0pt #1}\kern0.5pt}

\newcommand \cideal {\mathord\downarrow^{\kern-2pt\prec}\kern0.5pt }
\newcommand \pcideal[1] {\mathord\downarrow^{\kern-2pt\prec}_{\kern-1.0pt #1}\kern0.5pt}
\newcommand \oideal {\mathord\Downarrow\kern0.5pt }
\newcommand \poideal[1] {\mathord\Downarrow_{\kern-1.0pt #1}\kern0.5pt}

\newcommand \filter[1]{\mathord\uparrow\kern0.5pt  #1}

\newcommand \CONSPS {\textup{Con(\tbf{SPS})}}

\newcommand \tbdia {$\mathcal C_1$}

\newcommand \Enl [1] {\textup{Lit}(#1)}
\newcommand \tEnl  {\textup{Lit}}

\newcommand \LEnl [1] {\textup{LLit}(#1)}
\newcommand \REnl [1] {\textup{RLit}(#1)}

\newcommand \Lamp[1] {\textup{Lamp}(#1)}
\newcommand \rhbody {\boldsymbol\rho_{\textup{Body}}}
\newcommand \rhlrbody {\boldsymbol\rho_{\textup{LRBody}}}

\newcommand \rhfoot {\boldsymbol\rho_{\textup{foot}}}
\newcommand \rhinfoot {\boldsymbol\rho_{\textup{infoot}}}

\newcommand \Body [1] {\textup{Body}(#1)}
\newcommand \defiff {\overset{\textup{def}}{\iff}}
\DeclareMathOperator{\Con}{Con}
\newcommand{\tbf}{\textbf}
\newcommand{\set}[1]{\{#1\}}
\newcommand \possibly [1] {}
\newcommand \red[1]{{\textcolor{red}{#1}\color{black}}}
\newcommand \revised [1] {{\color{black}#1\color{black}}}

\newcommand \semmi[1]{}

\begin{document}

\title[\revised{\tbdia-diagrams of slim rectangular lattices}]
{\revised{\tbdia-diagrams of slim rectangular semimodular lattices permit quotient diagrams}} 

\author[G.\ Cz\'edli]{G\'abor Cz\'edli}
\email{czedli@math.u-szeged.hu}
\urladdr{http://www.math.u-szeged.hu/~czedli/}
\address{University of Szeged, Bolyai Institute. 
Szeged, Aradi v\'ertan\'uk tere 1, HUNGARY 6720}

\begin{abstract} 
\emph{Slim semimodular lattices} (for short, \emph{SPS lattices}) and \emph{slim rectangular lattices} (for short, \emph{SR lattices}) were introduced by 
G.\ Gr\"atzer and E.\ Knapp in 2007 and 2009. These lattices are necessarily finite and planar, and they have been studied in more then four dozen papers since 2007. They are best understood with the help of their \emph{\tbdia-diagrams}, introduced by the author in 2017.
For a  diagram $F$ of a finite lattice $L$ 
and a congruence $\alpha$ of $L$, we define the ``\emph{quotient diagram}'' $F/\alpha$ by taking the maximal elements of the $\alpha$-blocks and preserving their geometric positions. While $F/\alpha$ is not even a Hasse diagram in general, we prove that whenever $L$ is an SR lattice and $F$ is a \tbdia-diagram of $L$, then  $F/\alpha$ is a \tbdia-diagram of  $L/\alpha$, which is an SR lattice or a chain. 
The class of lattices isomorphic to the congruence lattices of SPS lattices is closed under taking
filters. We prove that this class is  closed under 
two more constructions, which are inverses of taking filters in  some sense; one of the two respective proofs relies on an inverse of the quotient diagram construction. 
\end{abstract}

\thanks{This research was supported by the National Research, Development and Innovation Fund of Hungary, under funding scheme K 134851.}

\subjclass {06C10\hfill{\red{\tbf{\datum}}}}

\dedicatory{{Dedicated to my grandchildren, P\'eter, Ad\'elia, Valentina Blanka, and Lili\'ana}}

\keywords{{Slim rectangular lattice, \tbdia-diagram, quotient diagram, slim  semimodular lattice, planar semimodular lattice, multifork, thrusting multiforks, congruence lattice,  join-irreducible element, lamp, order filter, Dioecious Maximal Elements Property}}

\maketitle

\section{Introduction}\label{S:Introduction}

\subsection{Slim planar semimodular lattices in lattice theory}
In 2007,  Gr\"atzer and Knapp \cite{GKn-I} introduced the concept of \emph{slim planar semimodular lattices}, \emph{SPS lattices} for short, as finite planar semimodular lattices without cover-preserving $M_3$-sublattices, where $M_3$ is the five-element modular lattice of length 2. Since 2007, more than four dozen papers have been devoted to or motivated by SPS lattices; see the appendix of 
http://arxiv.org/abs/2107.10202 or, for an updated version, see
\texttt{http://www.math.u-szeged.hu/\textasciitilde{}czedli/m/listak/publ-psml.pdf}. See also
 Section 2 of Cz\'edli and Kurusa \cite{CzGKA19} for the influence of these lattices on other fields of mathematics. Here we only mention that SPS lattices played the main role in Cz\'edli and Schmidt \cite{CzGSch-jordanh}, where we could add a uniqueness part to the 
the (group theoretical) Jordan--H\"older theorem from the nineteenth century; see Jordan \cite{Jordan} and H\"older \cite{Holder}. A large part of the research devoted to SPS lattices aims at the congruence lattices $\Con L$ of these lattices $L$.

\subsection{Our goals}\label{subsect:goals}
For a finite lattice $L$, its planar (Hasse) diagram $L^\sharp$, and a congruence $\alpha$ of $L$, we define the ``quotient diagram'' $L^\sharp/\alpha$ by 
 taking the maximal elements of the blocks of $\alpha$ and  keeping their geometric positions in the plane. Although this ``quotient diagram'' is not even a Hasse diagram of a poset in general, our main theorem asserts that $L^\sharp/\alpha$ is even a \tbdia-diagram (and so a planar lattice diagram) of $L/\alpha$ if $L^\sharp$ is a \tbdia-diagram of an SR lattice $L$\revised{; see Section \ref{sect:mainconc} for definitions}. We recall earlier tools from  \cite{CzGlamps} and slightly develop them further in order to prove the main theorem.

Let $\CONSPS$ denote the class of lattices that are isomorphic to congruence lattices of SPS lattices. We know from Cz\'edli and Gr\"atzer \cite[Theorem 1.2]{CzGGG3p3c} that $\CONSPS$ is closed, say, under taking finite direct products.  
\revised{Furthermore, after observing that $L/\alpha$ above is an SPS-lattice and using Gr\"atzer and Knapp \cite{GKnapp-III}, it is easy to show that  $\CONSPS$ is also closed under taking filters; see Observation \ref{obs:rePrk} for details. }
We are going to use some tools  from  \cite{CzGlamps}  to prove that $\CONSPS$ is closed under two additional constructions; these constructions are, in a sense, inverses of taking filters. 

The connection between the previous two paragraphs is close since the proof for one of the two new constructions mentioned above relies on an inverse of forming quotient diagrams. In fact, it is the study of $\CONSPS$ that led to the concept of quotient diagrams.

\subsection{Outline}\label{subsect:outline}
Section \ref{sect:mainconc} recalls the main concepts and most of the notations. Apart from some  lemmas like Lemmas \ref{lemma:cVcRtGl} and \ref{lemma:mgntBlht}, which could be of separate interest, we formulate the results of the paper  in Section \ref{sect:results}. Some easy results from Section \ref{sect:results} are proved in Section \ref{sect:easy}.
In Section \ref{sect:furthtls}, we recall some tools and concepts from earlier papers, mainly from  \cite{CzGlamps}, where the ``lamp technique'' was introduced. Section \ref{sect:diagrquotient} proves the main result, Theorem \ref{thm:mrKtpnMgsRflpbsmG}, on quotient diagrams. Under the name ``thrusting a lamp'' and (in a particular case) ``thrusting a multifork'', Section \ref{sect:thrust} introduces a  way of constructing a new SR lattice from a given one. 
Finally, using the tools recalled or developed in earlier sections, Section \ref{sect:tmprfs} proves Theorems \ref{thm:stnKzrzmG} and \ref{thm:hbhdlgnGtmDzrd} stating that \CONSPS\ is closed under two new constructions.

\section{Main concepts and notations}\label{sect:mainconc}
\subsection{Slim semimodular and slim rectangular lattices}\label{subsect:spssr}
For a lattice $L$, $\Jir L=(\Jir L;\leq)$ stands for the poset (that is, partially ordered set) of nonzero join-irreducible elements of $L$. The concept of slimness is extended for not necessarily semimodular lattices by Cz\'edli and Schmidt \cite{CzGSch-jordanh} as follows: a \emph{slim  lattice} is a finite  lattice $L$ such that $\Jir L$ is the union of two chains. It was proved in   \cite{CzGSch-jordanh} that slim lattices are planar. This allows us to omit the adjective ``planar'' in most of the rest of the paper. So let us summarize:
\begin{equation}
\text{slim semimodular lattice = slim planar semimodular lattice = SPS lattice.}
\end{equation}

Closely related to slim semimodular lattices, we  recall the definition of slim rectangular lattices, which were introduced by Gr\"atzer and Knapp \cite{GKnapp-III}. For a finite lattice $L$, $\Mir L$ stands for the poset of non-unit meet-irreducible elements of $L$. The elements of $\Jir L\cap\Mir L$ are called \emph{doubly irreducible}. Note that in any planar diagram of a slim semimodular lattice $L$, the doubly irreducible elements are on the boundary of $L$. Two elements of $L$ are \emph{complementary} if their meet is 0 while their join is 1.
By a \emph{slim rectangular lattice} 
\revised{(\emph{SR lattice} for short)} 
we mean a slim semimodular lattice that has exactly two doubly irreducible elements and these two elements are complementary. It is not hard to see and it was observed in Cz\'edli and Schmidt \cite[the paragraph above Theorem 12]{CzgScht-a-visual} that
\begin{equation}\left.
\parbox{10cm}{a semimodular lattice $L$ is a slim rectangular lattice if and only if $\Jir L$ is the union of two chains such that every element in one of these chains is incomparable with all elements of the other chain.}
\,\,\right\}
\label{eq:gnlTbnmSkrFgjvr}
\end{equation}
Since each element of $L$ is the join of some join-irreducible elements, it follows easily from \eqref{eq:gnlTbnmSkrFgjvr} that
\begin{equation}\left.
\parbox{10cm}{a semimodular lattice $L$ is a slim rectangular lattice if and only if there are chains $C$ 
and $C'$ in $L$  such that $\Jir L\subseteq C\cup C'$ and, for all $c\in C$ and $c'\in C'$, $c\wedge c'=0$.}
\,\,\right\}
\label{eq:nlSwhmpglKsztl}
\end{equation}

Slim rectangular lattices  (SR lattices for short) are in connection with slim semimodular lattices (SPS lattices for short) in several ways. Namely,  SR lattices are the ``building stones'' of SPS lattices by Cz\'edli and Schmidt
\cite{CzGSchT--patchwork}, they are extremal SPS lattices by \cite{CzGspsabsretracts}, 
 and they are the key to draw nice and useful diagrams by \cite{CzGdiagrectext}. Furthermore,  the study of congruence lattices of SPS lattices reduces to that of congruence lattices of SR lattices by Gr\"atzer and Knapp \cite{GKnapp-III}, where it is proved that 
\begin{equation}\left.
\parbox{11cm}{whenever $K$ is a slim semimodular lattice with at least three elements, then there is a slim rectangular lattice $K'$ such that $\Con K=\Con{K'}$.}
\,\,\right\}
\label{eq:mzdlmCltNtrGls}
\end{equation} 
Here for $L\in\set{K,K'}$ (and also for any lattice $L$), 
$\Con L=(\Con L;\subseteq)$ denotes the \emph{congruence lattice} of $L$. In 2016, Gr\"atzer \cite[Problem 24.1]{gGCFL2} and \cite{gG-cong-fork-ext} asked for a description of the congruence lattices of slim semimodular lattices. Since these congruence lattices, which are finite distributive lattices, are completely described by their posets  $\Jir{\Con L}$ of join-irreducible elements and these posets are simpler structures than the congruence lattices $\Con L$ in question, it is reasonable to focus on the properties of these posets. Indeed, the known properties of the congruence lattices $\Con L$ of SPS lattices $L$ have been proved and most of them have been formulated in terms of $\Jir{\Con L}$; see \cite{CzG:aNote14},  \cite{CzGlamps}, \cite{CzG-DCElamps},   Cz\'edli and Gr\"atzer \cite{CzGGG3p3c}, and
Gr\"atzer \cite{gG-cong-fork-ext} and \cite{gG-VIII-conSPS}.

\subsection{Diagrams}\label{subsect:diagrs}
Assume that we have a planar diagram of a finite lattice.
Given a usual coordinate system of the plane, the 
lines or line segments parallel to the vectors $(1,1)$ and $(1,-1)$ are said to be of \emph{normal slopes}. So there are two normal slopes. A line or line segment is \emph{precipitous} if it is vertical (that is, parallel to the $y$-axis) or it is parallel to a vector $(1,q)$ with $|q|>1$, where $q$ belongs to $\RR$, the set of real numbers. 
Equivalently, a line (segment) $\ell$ is precipitous if, denoting by $\angle(x,\ell)$ the angle from the $x$-axis to $\ell$, we have that $\pi/4 < \angle(x,\ell)< 3\pi/4$. 
Following  \cite[Definition 2.1]{CzGlamps}, originally \cite{CzGdiagrectext}, we need the following definition.

\begin{definition}\label{def:SRdiagram}
\revised{\textup{(A)}}
A planar diagram of a slim rectangular lattice $L$ is a \emph{\tbdia-diagram} if for any edge \revised{(in other words, prime interval)}{} $[p,q]$  of the diagram,
\begin{itemize}
\item if $p\in \Mir L$ and $p$ is not on the boundary, then $[p,q]$ is a precipitous, and
\item otherwise (that is, when $p\notin \Mir L$ or $p$ is on the boundary of the diagram), $[p,q]$ is of a normal slope.
\end{itemize}
\revised{\textup{(B)}} A diagram of a finite chain is a \emph{\tbdia-diagram} if all its edges are of normal slopes.
\end{definition}

We know from \cite{CzGdiagrectext} that each slim rectangular lattice has a \tbdia-diagram. Except for Figure \ref{fig15}, 
all \revised{\emph{lattice}} diagrams in this paper are \tbdia-diagrams of slim rectangular lattices. In what follows, the following convention is valid even when it is not mentioned again.

\begin{convention}\label{conv:tbdD} 
When we deal with a slim rectangular lattice, a \tbdia-diagram of this lattice is fixed. 
This allows us to use diagram-dependent concepts  and use the same notation for a slim rectangular lattice and its fixed \tbdia-diagram.
Whenever we construct a slim rectangular lattice $L^{(2)}$ from  a slim rectangular lattice $L^{(1)}$, then we also construct (and fix) the \tbdia-diagram of  $L^{(2)}$ from that of  $L^{(1)}$.  Sometimes, in text or notation, we do not make a distinction between a slim rectangular lattice and its \tbdia-diagram.
\end{convention}

Note that, apart from left-right symmetry and scaling, 
the \tbdia-diagram of a slim rectangular lattice is unique; the reader may (but need not) read the details of this uniqueness as a particular case of \cite[Theorem 5.5(iii)]{CzGdiagrectext}. Furthermore, denoting the \emph{left and right boundary chains} of a slim rectangular lattice $L$ by $\LBnd L$ and $\RBnd L$ and using the notation $\Bnd L:=\LBnd L\cup \RBnd L$ for the \emph{boundary} of $L$, none of 
$\Bnd L$ and the set $\set{\LBnd L,\RBnd L}$ depends on the \tbdia-diagram of $L$; see  \cite[Proposition 4.11]{CzGdiagrectext}.

\section{Results}\label{sect:results}

We are going to prove the results presented in this section in the subsequent sections. The following easy observation is known from the folklore; for convenience, we are going to present a short proof of it.

\begin{observation}\label{obs:wzGwk}
The quotient lattice $K/\alpha$ for a semimodular lattice $K$ and $\alpha\in \Con K$ need not be semimodular.
\end{observation}

As opposed to Observation \ref{obs:wzGwk}, the situation is more pleasant in the finite case.

\begin{observation}\label{obs:djbvmfsHgrG}
 For a  finite lattice $L$ and $\alpha\in\Con L$, the following three assertions hold.
\begin{enumeratei}
\item\label{obs:djbvmfsHgrGa} If $L$ is semimodular, then so is its quotient lattice $L/\alpha$.
\item\label{obs:djbvmfsHgrGb} If $L$ is a slim semimodular lattice, then so is $L/\alpha$.
\item\label{obs:djbvmfsHgrGc} If $L$ is a slim rectangular lattice and $L/\alpha$ is not a chain, then $L/\alpha$ is also a slim rectangular lattice.
\end{enumeratei}
\end{observation} 

Convention \ref{conv:tbdD} and Observation \ref{obs:djbvmfsHgrG}\eqref{obs:djbvmfsHgrGc} raise the question that how to obtain a \tbdia-diagram of $L/\alpha$ \emph{directly} from that of a slim rectangular lattice $L$ in an easy and \emph{natural way}.  Even though the existence of a \tbdia-diagram of 
$L/\alpha$ follows from  \cite[Theorem 5.5(ii)]{CzGdiagrectext} and Observation \ref{obs:djbvmfsHgrG}, the construction in \cite{CzGdiagrectext} does not benefit from the fixed \tbdia-diagram of $L$ and produces a \tbdia-diagram of $L/\alpha$ that has not  much to do with the \tbdia-diagram of $L$. 
By a \emph{diagram} we mean a Hasse diagram. Sometimes, ``Hasse'' is only added for emphasis.

\begin{definition}\label{def:qtnbdZcmRwhzm}
\textup{(A)} For a  lattice diagram $F$ and a subset $X$ of the vertex set of $F$, the ``\emph{subdiagram}'' (determined by) $X$ consists of the elements of $X$ as vertices (in their original geometric positions) and the straight line segments connecting $u,v\in X$ whenever $u\prec_X v$, that is, whenever $u<v$ in (the lattice determined by) $F$ but there is no $w\in X$ such that $u<w$ and $w<v$ hold in $F$. When we are sure that the  ``subdiagram''  is a  Hasse diagram of $X$ (with the ordering inherited from the lattice represented by $F$), we omit the quotient marks.

\textup{(B)} Let $F$ be a diagram of a finite lattice $L$, and let $\alpha\in\Con L$. The ``\emph{quotient diagram}'' $F/\alpha$ is the subdiagram of $F$ determined by the set of the  largest elements of the $\alpha$-blocks. When $F/\alpha$ is a Hasse diagram of the quotient lattice $L/\alpha$, we usually omit the quotient marks.
\end{definition}

\begin{figure}[ht] \centerline{ \includegraphics[scale=1.0]{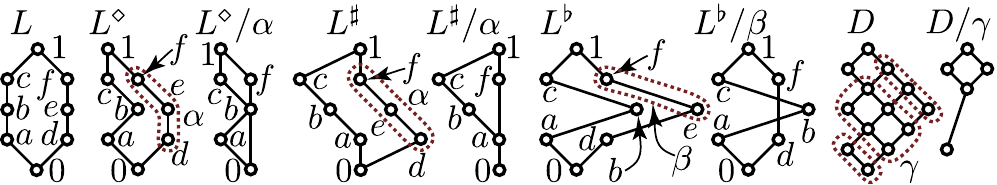}} \caption{Wrong ``quotient diagrams'', $D$, and $D/\gamma$}\label{fig15}
\end{figure} 

\begin{example}\label{example-nPntbd}
In Figure \ref{fig15}, $L$ on the left is also given by three additional planar diagrams, $L^\diamond$, $L^\sharp$, and $L^\flat$. The dotted ovals stand for the 
only non-singleton block of $\alpha$ and that of $\beta$. 
However, $L^\diamond/\alpha$ is \emph{not even a (Hasse) diagram} (because of the triangle $\set{0,a,b}$), $L^\sharp/\alpha$ is a diagram but \emph{not a diagram of} the quotient lattice $L/\alpha$ (because $a/\alpha \nleq f/\alpha$ in $L/\alpha$), and  $L^\flat/\beta$ is a \emph{non-planar} diagram of $L/\beta$. 
In case of $L^\diamond/\alpha$ and $L^\sharp/\alpha$, the problem is caused by an edge that goes through a ``forbidden'' vertex. 
Each of the congruences $\alpha$ and $\beta$ has a  special property since its maximal elements form a sublattice. Hence,  this example is also about ``subdiagrams'', not just ``quotient diagrams''.
\end{example}

\begin{example}\label{example-jnKRd}
In Figure \ref{fig6},  
the diagram of $\Shk 2$ is a \tbdia-diagram but its subdiagram $\Shk 2\setminus\set{\rcorner{\Shk 2}}$, which determines a slim rectangular sublattice, is not a \tbdia-diagram.
\end{example}

A number of proofs in \cite{CzGdiagrectext}, \cite{CzGlamps}, \cite{CzG-DCElamps}, \cite{CzGpropmeet}, \cite{CzGspsabsretracts}, \cite{CzG-lampstwoarx}, and Cz\'edli and Gr\"atzer \cite{CzGGG3p3c} benefit a lot from \tbdia-diagrams explicitly. Implicitly, so  did several papers even before the introduction of this concept.  Now, in view of
 Example \ref{example-nPntbd},  the theorem below reveals a new advantageous and special feature of \tbdia-diagrams.

\begin{theorem}\label{thm:mrKtpnMgsRflpbsmG} If $F$ is a \tbdia-diagram of a slim rectangular lattice $L$, then the quotient diagram $F/\alpha$ defined in Definition \ref{def:qtnbdZcmRwhzm} is a \tbdia-diagram of $L/\alpha$.
\end{theorem}

\begin{example} In spite of a result in Gr\"atzer and Knapp \cite{GKnapp-III}, which is cited here in \eqref{eq:mzdlmCltNtrGls},  Theorem \ref{thm:mrKtpnMgsRflpbsmG} does extend to slim semimodular lattices, not even to slim distributive lattices; this is witnessed by $D/\gamma$ in Figure \ref{fig15}.
\end{example}

Recall that  for an element $u$ in a poset (in particular, a lattice)  $P$, the \emph{principal ideal} $\set{x\in P: x\leq u}$ and the \emph{principal filter} $\set{x\in P: x\geq u}$ are denoted by $\ideal u=\pideal P \revised{u}$ and $\filter u=\pfilter P u$, respectively. 
A subset $X$ of $P$ is an \emph{order filter} or, in other words, an \emph{up-set} of $P$ if for every  $u\in X$,  $\pfilter P u\subseteq X$. The poset of order filters of $P$ will be denoted by $\Ofilt P=(\Ofilt P;\subseteq)$. 
Let us emphasize that  $\emptyset\in \Ofilt P$. For a finite lattice $L$,  the \emph{$($lattice$)$ filters} are the same as the  \emph{principal filters}; the poset they form will be denoted by $\Lfilt L=(\Lfilt L;\subseteq)$. 
Even if $L$ is a chain,  $\Ofilt L\neq\Lfilt L$ since $\emptyset\notin\Lfilt L$. For a finite distributive lattice $D$, there is a canonical correspondence between the (lattice) filters of $D$ and the order filters of $\Jir D$. Namely, we have the following observation.

\begin{observation}\label{obs:rdltfLtrs} 
For a finite distributive lattice $D$, 
the poset $(\Lfilt D;\subseteq)$ of the filters of $D$ and the poset $(\Ofilt{\Jir D};\subseteq)$ of the order filters of $\Jir D$ are isomorphic. Namely, the functions
\begin{align*}
\psi&\colon\Lfilt D\to\Ofilt{\Jir D}\text{ defined by }X\mapsto \Jir D\, \setminus \,\pideal D {\textstyle{\mathord{\bigwedge}} X}\,\,\text{ and }
\\
\gamma&\colon\Ofilt{\Jir D}\to\Lfilt D\text{ defined by }
Y\mapsto  \pfilter D{\textstyle{\mathord{\bigvee}}\bigl(\Jir D\setminus Y\bigr)}
\end{align*}
are reciprocal order isomorphisms. Furthermore, for $X\in\Lfilt D$,  $X$ is a lattice of its own right and so $(\Jir X;\leq)$ is a poset, $(\psi(X);\leq)$ (as an order filter of $\Jir D$) is also a poset, and these two posets are isomorphic, that is, 
$(\Jir X;\leq)\cong (\psi(X);\leq)$.
\end{observation}

Since we have not found this easy statement in the literature, we prove it in Section \ref{sect:easy}.
Note that the posets $\Jir X$ and $\psi(X)$ above are different subsets of $D$ in general. 
Unless it is explicitly stated otherwise, a poset is nonempty by definition. 
When we want to allow the empty set as a poset, then we will speak about a ``\emph{possibly empty poset}". Similarly, we can call $\emptyset$ the ``empty poset''. For example, if $L$ is the one-element lattice, then $\Jir L$ is the empty poset.

\begin{definition}\label{def:ssrrepr} \textup{(A)} A finite distributive lattice $D$ is \emph{CSPS-representable} or \emph{CSR-representable} if  $D\cong\Con L$ for some slim semimodular or for some slim rectangular lattice $L$,  respectively.

\textup{(B)} A finite and possibly empty poset $P=(P;\leq)$ is \emph{JCSPS-representable} or \emph{JCSR-representable} if 
$(P;\leq)\cong (\Jir{\Con L};\leq)$ for some slim semimodular or for some slim rectangular lattice $L$,  respectively.
\end{definition}

The adjective ``distributive"   in Part (A) is only for emphasis since, for any lattice $L$,  $\Con L$ is known to be distributive by a 1942 result of Funayama and Nakayama. 
The singleton lattice witnesses that the empty poset is JCSPS-representable. 

Based on Definition \ref{def:ssrrepr}, we have the following statement.

\begin{observation}\label{obs:rePrk} Let $D$ and $P$ be 
a finite distributive lattice and a  finite poset, respectively. 

\textup{(A)} If $D$ is CSPS-representable, then so are its filters and, furthermore, any at least three-element filter of $D$ is CSR-representable.

\textup{(B)}
If $F$ is an order filter of a JCSPS-representable poset, then $F$  is also JCSPS-representable. 
If $F$ is an at least two-element order filter of a JCSR-representable poset, then $F$ is also JCSR-representable.
\end{observation}

This observation implies the following statement trivially.

\begin{corollary}\label{cor:jlMzs} 
Let $F$ be a finite poset,  and let $E$ be a finite distributive lattice.

\textup{(A)} 
 If $F$ is not  JCSPS-representable, then every JCSPS-representable poset $P$ has the property that no order filter of $P$ is order isomorphic to $F$. 

\textup{(B)} If $E$ is not CSPS-representable, then no CSPS-representable finite distributive lattice has a lattice filter isomorphic to $E$.  
\end{corollary}

Gr\"atzer \cite{gG-VIII-conSPS} proved that the two-element chain as a poset is not JCSPS-rep\-resen\-table. Combining this fact with Corollary \ref{cor:jlMzs}, we immediately obtain a new proof of Corollary 3.5 of Cz\'edli \cite{CzGlamps}, which asserts the following.

\begin{corollary}[Dioecious Maximal Elements Property]\label{cor:Dioecious}
The two-element chain is not an order filter of a  JCSPS-representable poset.
\end{corollary}

\revised{Combining \eqref{eq:mzdlmCltNtrGls} and Corollary \ref{cor:Dioecious}, we obtain the following remark.
\begin{remark}\label{rem:cskGgJzSfHm}
A possibly empty finite poset $P$ is JCSPS-representable if and only if $|P|\in\set{0,1}$ or $P$ is JCSR-representable. Similarly, a finite poset $Q$ is JCSR-representable if and only if it is JCSPS-representable and $|Q|\geq 2$.
\end{remark}
}

Although several properties of JCSPS-representable \revised{\emph{posets}} are already known, see \cite{CzGlamps}, \cite{CzG-DCElamps}, and Cz\'edli and Gr\"atzer \cite{CzGGG3p3c}, and most of these properties make it easy to find a finite poset $F$ that is not JCSPS-representable, the result obtained from  such an $F$ by Corollary \ref{cor:jlMzs} is already in \cite{CzGlamps}, \cite{CzG-DCElamps} or \cite{CzGGG3p3c}. However, the algorithm given 
Cz\'edli \cite{CzG-lampstwoarx} could lead to new results in the future. 

In spite of many known properties of CSPS-representable  \revised{\emph{lattices}} proved in  \cite{CzG:aNote14}, \cite{CzGlamps}, \cite{CzG-DCElamps}, Cz\'edli and Gr\"atzer \cite{CzGGG3p3c}, and Gr\"atzer \cite{gG-cong-fork-ext} and \cite{gG-VIII-conSPS}, no (finite) characterization of these lattices is known in the language of lattice theory. However, Cz\'edli and Gr\"atzer \cite[Theorem 1.2]{CzGGG3p3c} and  Observation \ref{obs:rePrk} show that the class $\CONSPS$ is closed under some constructions. In order to state, in terms of JCSPS-representability, that this class is closed under two additional constructions,  we give two definitions below and then two theorems based on these definitions.

\begin{figure}[ht] \centerline{ \includegraphics[width=0.85\textwidth]{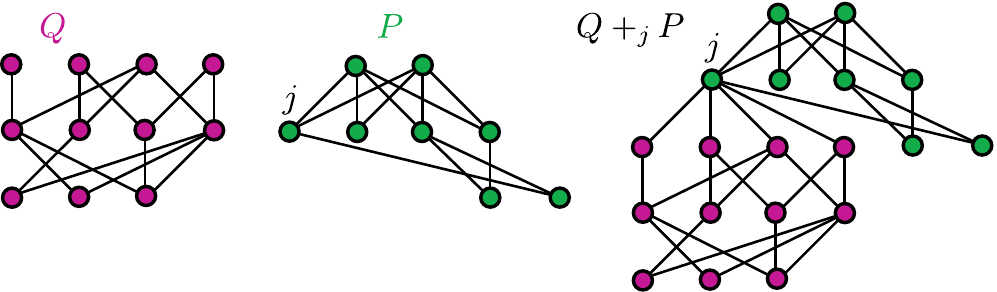}} \caption{\revised{Ordinal sum at an element:} illustrating the $Q\jsum j P$ construct}\label{fig1}
\end{figure}

\begin{definition}[Ordinal sum at an element; see Figure \ref{fig1}]\label{def:kRmvlMgrdh}
Let $P$ and $Q$ be  posets such that  $P\cap Q=\emptyset$, and let  $j\in P$. On the support set $P\cup Q$, we define a poset, denoted by $Q\jsum j P$, by letting  
\begin{equation}
x\leq y\defiff
\begin{cases}
x,y\in P\text{ and }x\leq y\text{ in }P,\text{ or}\cr
x,y\in Q\text{ and }x\leq y\text{ in }Q,\text{ or}\cr
x\in Q,\text{ }y\in P,\text{ and }j\leq y \text{ in }P.
\end{cases}
\end{equation}
\end{definition}

\begin{theorem}\label{thm:stnKzrzmG}
If $P$ and $Q$ are disjoint finite posets, $j\in P$ is not a maximal element of $P$, and both $P$ and $Q$ are JCSR-representable, then $Q\jsum j P$ is also JCSR-representable.
\end{theorem}

\begin{figure}[ht] \centerline{ \includegraphics[scale=1.0]{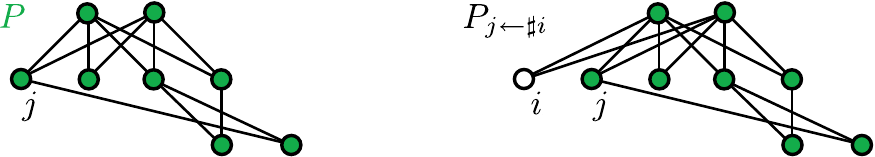}} \caption{\revised{Adding a brother:} illustrating the $\brosum P j i$ construct}\label{fig7}
\end{figure}

\revised{
If $P=\set{j}$ and $|Q|=|P|=1$, then  Corollary \ref{cor:Dioecious} and Remark \ref{rem:cskGgJzSfHm} show that $P$ and $Q$ are JCSPS-representable but $Q\jsum j P$ is not. However, Remark \ref{rem:cskGgJzSfHm} would allow us to reword Theorem  \ref{thm:stnKzrzmG} in terms of JCSPS-representability.}

The nickname of the following definition comes from considering two elements with exactly the same covers (= parents) to be brothers; however, note that  any two maximal (that is, ``parent-less'') elements are  brothers in this terminology.

\begin{definition}[``Adding a brother'']\label{def:kbClbtsmThkm}
For a finite poset $P$ and $j\in P$, let $U_j$ be the set of covers of $j$ in $P$; see on the left of Figure \ref{fig7} for an example with $|U_j|=2$.  Pick an element $i$ outside $P$.
On the support set $P\cup \set i$, we define a poset 
$\brosum P j i$ by  its covering relation
\begin{equation}
x \prec y\defiff
\begin{cases}
x,y\in P\text{ and }x\prec y\text{ in }P,\text{ or}\cr
x=i \text{ and }y\in U_j.
\end{cases}
\end{equation}
\end{definition}

\begin{theorem}\label{thm:hbhdlgnGtmDzrd}
If $P$ is a finite poset, $j\in P$, $i\notin P$, and $P$ is JCSR-representable, then 
$\brosum P ji$ is also JCSR-representable.
\revised{Similarly, if $P$ is JCSPS-representable, then so is $\brosum P ji$.}
\end{theorem}

Next, we summarize what we know about adding a new minimal element  to a JCSPS-representable poset.  
Later,   \eqref{eq:mlHtntbkNcZcs} and  Lemmas \ref{lemma:mfXzsrjsmG} and \ref{lemma:vltfLm} will indicate that if we want to build a JCSPS representable poset from smaller ones by adding new elements one by one, then it is natural to add new \emph{minimal} elements.

\begin{remark}\label{remark:wWrkfJwzr} 
Assume that $P$ is a JCSPS-representable poset and we intend to define a larger poset $P'$ by adding a new \emph{minimal} element $i$  to $P$. Then $P'$ is fully described by $P$ and the set $U$ of the upper covers of $i$ in $P'$, where $U$ is a possibly empty antichain of $P$. In terms of $P$ and $|U|$, what can we say about the JCSPS-representability of $P'$? We know the following. 

\textup{(a)} If $|U|=0$, then $P'$ is JCSPS-representable by Theorem \ref{thm:hbhdlgnGtmDzrd}. 

\textup{(b)}  If $|U|=1$, then $P'$ is JCSPS-representable since the only element of $\set{j}:=U$ is not a maximal element of $P$ by 
the Dioecious Maximal Elements Property, see \cite[Corollary 3.5]{CzGlamps} or Corollary \ref{cor:Dioecious} here, whereby  the $|Q|=1$ case of Theorem \ref{thm:stnKzrzmG} applies.

\textup{(c)}  For $|U|=2$, too little is known.  If $U$ is also the set of covers of an old element $j\in P$, then the answer is affirmative by Theorem \ref{thm:hbhdlgnGtmDzrd}. However, there are many cases where the answer is negative since the targeted new minimal element would violate known properties of JCSPS-representable posets; see \cite[Lemma 3.9]{CzGlamps}, \cite[Theorem 2.8]{CzG-DCElamps}, and 
Cz\'edli and Gr\"atzer \cite[Main Theorem]{CzGGG3p3c}.  
For example, with  $R$ given in \cite[Figure 9]{CzGlamps},
$P:=R\setminus\set w$ is JCSPS-representable but $P'=P\cup\set w=R$, in which $w$ has exactly two covers,  is not.

\textup{(d)} For $|U|\geq 3$, $P'$ is not JCSPS-representable
by the  Two-Cover Theorem proved in Gr\"atzer \cite{gG-cong-fork-ext} (see also Remark \ref{rem:cSgslVdkRzDrg} here).
\end{remark}

\section{Easy proofs}\label{sect:easy}
In this section, we prove those statements of Section \ref{sect:results} that do not need tools based on lamps.

\begin{proof}[Proof of Observation \ref{obs:wzGwk}]
Let $K=\set{o,i}\cup(\set{0,1}\times \QQ)$ where $\QQ$ is the set of rational numbers. Define the ordering ``$\leq$'' of $K$ so that $o$ and $i$ are the least element and the largest element of $K$, respectively, and for $(i,x),(j,y)\in \set{0,1}\times \QQ$, $(i,x)\leq(j,y)$ if and only if $i=j$ and $x\leq y$.
Since there is no covering pair of elements, $K$ is semimodular. 
However, if $\alpha\in\Con K$ is the congruence with blocks $\set o$, $\set i$, 
$\set{(0,x): \sqrt 2>x\in\QQ}$, $\set{(0,x): \sqrt 2<x\in\QQ}$, and $\set 1\times\QQ$, then $K/\alpha\cong N_5$ is not semimodular. 
\end{proof}

\begin{proof}[Proof of Observation \ref{obs:djbvmfsHgrG}]
Let $L$ be a finite semimodular lattice and $\alpha\in\Con L$. By the \revised{(only)} lemma proved in Cz\'edli and Schmidt \cite{CzGSch-howtderiv}, $\alpha$ is cover-preserving. 
The $\alpha$-block $\set{v\in L: (u,v)\in\alpha}$ of an element $u\in L$ will be denoted by $u/\alpha$.

To prove Part \eqref{obs:djbvmfsHgrGa}, assume that $x/\alpha \prec y/\alpha$ and $x/\alpha <z/\alpha$ in $L/\alpha$. By finiteness, the $\alpha$-blocks are intervals. This allows us to assume that $x$ is the largest element of its $\alpha$-block, $x/\alpha$. 
We can also assume that $x< y$ and $x<z$ since otherwise we can replace $y$ and $z$ by $x\vee y$ and $x\vee z$, respectively. Indeed, $(x\vee y)/\alpha=x/\alpha \vee y/\alpha=y/\alpha$, and similarly for $z$. Next, by modifying $y$ if necessary,
we can assume that $y$ is the smallest element of ${\filter x}\cap(y/\alpha)$. Now let $t\in[x,y]$. Since $x/\alpha\leq t/\alpha \leq y/\alpha$ and $x/\alpha \prec y/\alpha$, it follows that either $t/\alpha=x/\alpha$ or $t/\alpha=y/\alpha$.
If the first alternative holds, then $t\in x/\alpha$ and the maximality of $x$ in its block give that $t=x$. In case of the second alternative, we have that $y\leq t$ since $t\in  \filter x\cap(y/\alpha)$, whereby $t=y$. 
We have shown that $x\prec y$. By the semimodularity of $L$, 
$x\vee z\preceq y\vee z$. Using that $\alpha$ is cover-preserving, we obtain that $x/\alpha\vee z/\alpha =(x\vee z)/\alpha \preceq (y\vee z)/\alpha = y/\alpha\vee z/\alpha$. Thus, $L/\alpha$ is semimodular, proving Part \eqref{obs:djbvmfsHgrGa}.

To prove Part \eqref{obs:djbvmfsHgrGb}, note that $\Jir L=C\cup C'$ with chains $C$ and $C'$ since $L$ is slim. Let $C/\alpha:=\set{x/\alpha: x\in C}$ and 
$C'/\alpha:=\set{y/\alpha: y\in C'}$; they are also chains.  
If $u/\alpha\in L/\alpha$, then  $u\in C\cup C'$ or $u$ is of the form $x\vee y$ with $x\in C$ and $y\in C'$, whereby 
$u/\alpha\in C/\alpha\cup C'/\alpha$ or $u/\alpha=x/\alpha\vee y/\alpha$. This shows that $C/\alpha\cup C'/\alpha$, which is the union of two chains, generates the join-semilattice $(L/\alpha, \vee)$. Thus, $\Jir{L/\alpha}\subseteq C/\alpha\cup C'/\alpha$ shows that $\Jir{L/\alpha}$ is the union of two chains. Hence $L/\alpha$ is slim, proving Part \eqref{obs:djbvmfsHgrGb}.

To prove Part \eqref{obs:djbvmfsHgrGc}, assume that $L$, $C$ and $C'$ are as in \eqref{eq:nlSwhmpglKsztl} and $L/\alpha$ is not a chain.
We know from the previous paragraph that $\Jir{L/\alpha} \subseteq C/\alpha\cup C'/\alpha$. We let $C_1:=\set{c\in C: c/\alpha\in \Jir{L/\alpha}}$ and $C'_1:=\set{c'\in C': c'/\alpha\in \Jir{L/\alpha}}$. Then $C_1/\alpha$ and $C'_1/\alpha$ are chains and  $\Jir{L/\alpha} = {C_1/\alpha}\cup {C'_1/\alpha}$. None of $C_1$ and $C_1'$ is empty since otherwise $L/\alpha$ would be a chain. Take an element of $C_1/\alpha$  and that of $C'_1/\alpha$. These elements are of the form $c/\alpha$ and $c'/\alpha$ with $c\in C_1$ and $c'\in C_1'$, whence  their meet is
$c/\alpha\wedge c'/\alpha=(c\wedge c')/\alpha=0/\alpha$. 
Therefore, $L/\alpha$ is a slim rectangular lattice by \eqref{eq:nlSwhmpglKsztl}. This proves Part \eqref{obs:djbvmfsHgrGc} and completes the proof of Observation \ref{obs:djbvmfsHgrG}.
\end{proof}

\begin{proof}[Proof of Observation \ref{obs:rdltfLtrs}]
For a poset $P$, let $\Oid P$ stand for the set of \emph{order ideals} of $P$. That is, $X\in\Oid P$ if and only if $(\forall u\in X)(\pideal P u\subseteq X)$ and $X\subseteq P$. The set $\set{\pideal D u: u\in D}$ of (lattice) ideals of $D$ will be denoted by $\Lid D$. 
By the well-known structure theorem of finite distributive lattices, see e.g. Gr\"atzer \cite[Theorem 107]{r:Gr-LTFound}, 
the functions
\allowdisplaybreaks{%
\begin{align*}
f_1&\colon \Oid{\Jir D}\to D,\text{ defined by }X\mapsto \bigvee X\text{, and}\cr
f_1^{-1}&\colon D\to \Oid{\Jir D},\text{ defined by }u\mapsto \Jir D\cap\pideal D u
\end{align*}}%
are reciprocal bijections. Consider two \revised{more auxiliary} functions,
\allowdisplaybreaks{%
\begin{align*}
f_2&\colon \Ofilt{\Jir D}\to \Oid{\Jir D},\text{ defined by }X\mapsto \Jir D\setminus X,\text{ and}\cr
f_3&\colon D\to \Lfilt D,\text{ defined by }y\mapsto \pfilter D y.
\end{align*}}%
Clearly, $f_2$ and $f_3$ are functions as specified, and they are bijections with inverses
\allowdisplaybreaks{%
\begin{align*}
f_2^{-1}&\colon \Oid{\Jir D}\to \Ofilt{\Jir D},\text{ defined by }X\mapsto \Jir D\setminus X,\text{ and}\cr
f_3^{-1}&\colon \Lfilt D\to  D,\text{ defined by }Y\mapsto \bigwedge Y.
\end{align*}}%
We compose functions from right to left, that is, $(g_1\circ g_2)(x)=g_1(g_2(x))$. 
Clearly, $\gamma= f_3 \circ f_1 \circ f_2$. Using the identity 
$X\setminus(X\cap Y)=X\setminus Y$, we obtain that $\psi=f_2^{-1} \circ f_1^{-1} \circ f_3^{-1}$. Thus, $\psi$ and $\gamma$ are reciprocal bijections. Since $f_1$ and $f_1^{-1}$ are isotone (order-preserving) while $f_2$, $f_2^{-1}$, $f_3$, $f_3^{-1}$ are antitone (order-reversing), we obtain that $\psi$ and $\gamma$ are order-isomorphisms, as required.

Next, recall from the folklore that for a finite distributive lattice $D$, 
\begin{equation}\left.
\parbox{8.4cm}{if $p\in \Jir D$, $x_1,\dots,x_n\in D$, and $p\leq x_1\vee\dots\vee x_n$, then there is an $i\in\set{1,\dots,n}$ such that $p\leq x_i$.}\,\,
\right\}
\label{eq:nzjlMkgvBfztTn}
\end{equation} 
Indeed, the inequality we assume in \eqref{eq:nzjlMkgvBfztTn} yields that $p=p\wedge(x_1\vee\dots\vee x_n)=(p\wedge x_1)\vee\dots\vee(p\wedge x_n)$, and we obtain an $i$ by the join-irreducibility of $p$. 

To prove the last part of Observation \ref{obs:rdltfLtrs}, assume that $X\in\Lfilt D$. Let $u:=\bigwedge X$ and $Y:=\psi(X)$. Clearly, $X=\pfilter L u$ and $Y=\set{p\in\Jir D: p\nleq u}$. 
We claim that the function $\mu\colon Y\to \Jir X$ defined by $\mu(p):=u\vee p$ is an order isomorphism. 
For $p\in Y$, $\mu(p)$ is clearly in $X$. Furthermore, 
$\mu(p)\neq 0_X=u$ since $p\nleq u$. 
Assume that $a_1,a_2\in X$ such that $\mu(p)=a_1\vee a_2$. 
Then $p\leq \mu(p)=a_1\vee a_2$ and \eqref{eq:nzjlMkgvBfztTn}
give an $i\in\set{1,2}$ with $p\leq a_i$. Hence, 
$\mu(p)=u\vee p\leq u\vee a_i=a_i\leq\mu(p)$ yields that $\mu(p)=a_i$.  Thus, $\mu(p)\in\Jir X$ and $\mu$ is indeed a function from $Y$ to $\Jir X$. 

Next, let $z\in\Jir X$; note that $z>0_X$. As any element of $D$,  $z$ is the join of some elements of $\Jir D$. Since $z \nleq 0_X=u$, there must be joinands outside $\pideal D u$, that is, joinands belonging to $Y$. Hence, there are $p_1,\dots, p_k\in Y$
such that $z=u\vee p_1\vee\dots\vee p_k$ and $k\geq 1$. Thus, $z\in\Jir X$ and 
$z=(u\vee p_1)\vee\dots\vee(u\vee  p_k)$ imply that, for some $i\in\set{1,\dots,k}$, $z=u\vee p_i=\mu(p_i)\in \mu(Y)$, which shows that $\mu$ is surjective.
If $p,q\in Y$ such that $\mu(p)=\mu(q)$, then $p\leq \mu(p)=\mu(q)=u\vee q$ together with $p\nleq u$ and \eqref{eq:nzjlMkgvBfztTn} give that $p\leq q$. Since $q\leq p$ follows similarly, we obtain that $\mu$ is injective. Hence, $\mu$ is a bijection.
Clearly, $\mu$ is order-preserving. For  $p,q\in Y$ such that  $\mu(p)\leq\mu(q)$, we have that 
$p\leq \mu(p)\leq  \mu(q)=u\vee q$,  whereby $p\nleq u$ and \eqref{eq:nzjlMkgvBfztTn} give that $p\leq q$. Hence, the inverse of $\mu$ is also order-preserving. Thus, $\mu$ is an order isomorphism, completing the proof of Observation \ref{obs:rdltfLtrs}
\end{proof}

\begin{proof}[Proof of Observation \ref{obs:rePrk}]
To prove part (A), assume that $F$ is a filter of a CSPS-representable finite distributive lattice $D$. Take a slim semimodular lattice $L$ such that $D\cong \Con L$. 
By the Correspondence Theorem, see e.g. Burris and Sankappanavar \cite[Theorem 6.20]{burrissank}, there is a congruence $\alpha\in\Con L$ such that $F\cong \Con(L/\alpha)$ (as lattices).
But $L/\alpha$ is a slim semimodular lattice by Observation \ref{obs:djbvmfsHgrG}\eqref{obs:djbvmfsHgrGb}. Hence $F$ is CSPS-representable, as required.  Furthermore, if $|F|\geq 3$, then it is CSR-representable by \eqref{eq:mzdlmCltNtrGls}, and we have shown Part (A). 

To prove Part (B),  let $F$ be an order filter of a JCSPS-representable poset. We can assume that $F$ is an order filter of $\Jir{\Con L}$ for a slim semimodular lattice $L$. Applying Observation \ref{obs:rdltfLtrs} to
$D:=\Con L$, we obtain that $E:=\gamma(F)$ is a filter of  $D=\Con L$.  Since $D$ is CSPS represent\revised{able} by its definition,  the already proven Part (A) of Observation \ref{obs:rePrk} yields a slim semimodular lattice $K$ such that $E\cong \Con K$. Hence, using Observation \ref{obs:rdltfLtrs} at $\cong'$ and $='$, we have that
\begin{equation}\Jir{\Con K}\cong\Jir E \cong' \psi(E)=\psi(\gamma (F))=' F.
\label{eq:mkRhHlvlsg}
\end{equation}
Thus, $F$ is JCSPS-representable. If $|F|\geq 2$, then $|\Jir{\Con K}|\geq 2$ by \eqref{eq:mkRhHlvlsg}. Thus,  $|\Con K|\geq 3$.
So, $|K|\geq 3$. By \eqref{eq:mzdlmCltNtrGls}, $\Con K\cong \Con{K'}$ for a slim rectangular lattice $K'$. Combining this isomorphism with \eqref{eq:mkRhHlvlsg}, we have that 
$F\cong\Jir{\Con{K'}}$,  showing the JCSR-representability of $F$ and completing the proof of Observation \ref{obs:rePrk}.
\end{proof}

\begin{figure}[ht] \centerline{ \includegraphics[width=0.95\textwidth]{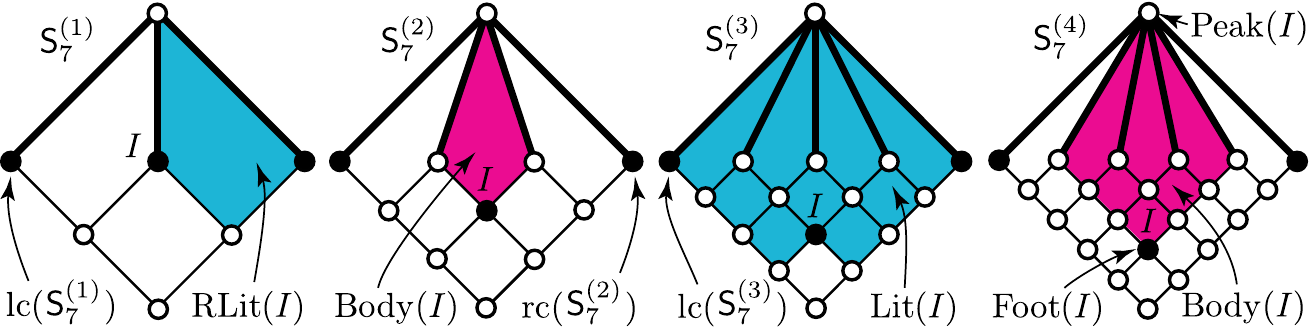}} \caption{$\Shk n$ for $n\in\set{1,2,3,4}$}\label{fig6}
\end{figure}

\begin{figure}[ht] \centerline{ \includegraphics[scale=0.95]{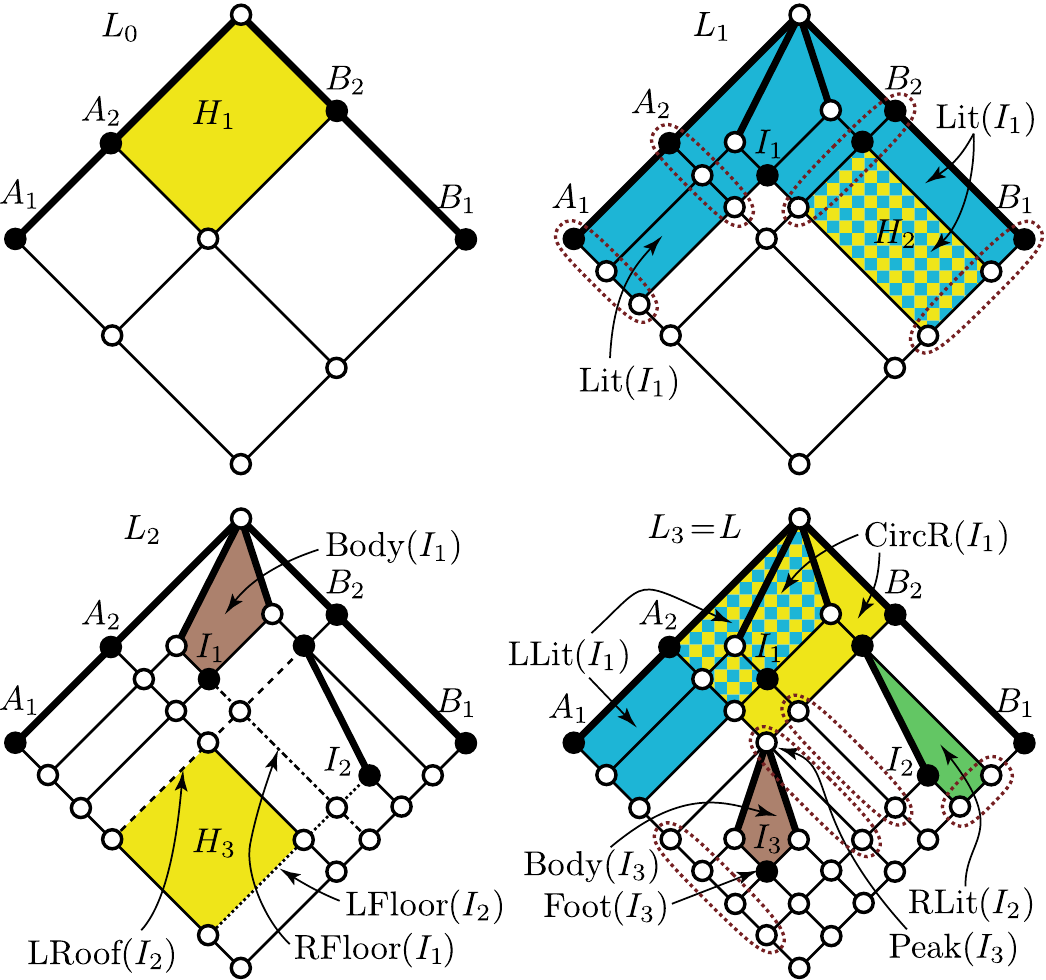}} \caption{A standard multifork sequence and some geometric objects}\label{fig3}
\end{figure}

\section{Further definitions and tools}\label{sect:furthtls}

In this section, we collect several concepts and  earlier results that we need to prove Theorems \ref{thm:mrKtpnMgsRflpbsmG}, \ref{thm:stnKzrzmG}, and \ref{thm:hbhdlgnGtmDzrd}. This ``mini-collection'' as 
well as some statements, like Lemma \ref{lemma:cVcRtGl}, in subsequent sections could be useful in future research.

Let $L$ be a slim rectangular lattice.
The rectangle formed by $\LBnd L$ and $\RBnd L$ is called the \emph{full geometric rectangle} of $L$; it is denoted by $\FullRect L$. It is always a geometric rectangle with sides of normal slopes since a \tbdia-diagram of $L$ is fixed; see Convention \ref{conv:tbdD}. 
An \emph{$L$-compatible line segment} is a geometric line segment that is the union of some edges of the diagram. In particular, every edge is an $L$-compatible line segment.
Let  $a<b$ in $L$. If $C'$ and $C''$ are maximal chains of the interval $[a,b]$, then the collection of lattice elements of $[a,b]$ that are on the right of $C'$ and, simultaneously, on the left of $C''$ 
\begin{equation}
\text{is called a \emph{$($lattice$)$ region} of $L$.}
\label{eq:grnBndmlszGn}
\end{equation}
Note that $C'$ and $C''$ are subsets of this lattice region. This concept is borrowed from Kelly and Rival \cite{KR75}, where it is proved that lattice regions are cover-preserving sublattices and intervals are lattice regions. 
Our definition of a region relies on the fixed \tbdia-diagram of $L$. 
For regions and other concepts that were or will be defined later (like $\Bnd L$, territory, $\Enl I$, etc.), the following convention applies.

\begin{convention}If we define a concept $H$ based on the geometric properties of the \tbdia-diagram of our
slim rectangular lattice $L$, then $H$ has \revised{a} double meaning. First, $H$ is a set of certain elements and edges of $L$. Second, $H$ is a geometric object consisting of line segments \revised{and (usually infinitely many)} geometric points.  When the distinction between these two meanings is important, the context will make it clear which one is understood. In particular, the terms ``lies'' and ``in the plane'' always carry a geometric meaning.
Sometimes, we explicitly say that the meaning is geometric or we use a different notation (like $\Body I$ later) for the geometric meaning. For example, $H$ is understood lattice theoretically in the text ``$x\in L$ is in $H$'' and in Kelly and Rival's previously-mentioned result. On the other hand, $H$ is meant geometrically in the text ``a part of an edge lies in $H$''. 
\end{convention}

Non-chain 4-element lattice regions are  called \emph{$4$-cells}. If all minimal non-chain lattice regions are 4-cells, then $L$ is called a \emph{$4$-cell lattice}. By
Gr\"atzer and Knapp {\cite[Lemmas 4 and 5]{GKn-I} and \cite{GKnapp-III}}, we know that for a planar lattice $L$ (which is finite by definition),
\begin{equation}\left.
\parbox{10.2cm}{if  $L$ is a $4$-cell lattice, no two distinct $4$-cells have the same bottom, $L$ has exactly two doubly irreducible elements, \revised{ and these two elements are complementary,} then $L$ is a slim rectangular lattice.}\,\,\right\}
\label{eq:mnHhwlCh}
\end{equation}
Conversely, a slim rectangular lattice is a 4-cell lattice \revised{with the above properties}.

For an interval $H$, the least element and the largest element of $H$ will be denoted by $\Foot H=0_H$ and $\Peak H=1_H$. 
(We prefer the notation $\Foot H$ and $\Peak H$, introduced in \cite{CzGlamps} for some special intervals since they are more readable then $0_H$ and $1_H$ if $H$ is subscripted like $H_i$.) 
\begin{equation}\left.
\parbox{10cm}{A \emph{territory} of a slim rectangular lattice $L$  with a fixed \tbdia-diagram is   a not necessarily convex polygon together with  its inner points bordered by $L$-compatible line segments.}\,\,\right\}
\label{eq:hsGsmfKrnnSlk} 
\end{equation}
Comparing \eqref{eq:grnBndmlszGn} and \eqref{eq:hsGsmfKrnnSlk}, we can see that  regions  are territories but not conversely. 
Note that an edge in itself is a territory (and also a region), so the area (two-dimensional Lebesgue measure) of a territory can be 0.
\begin{equation}\left.
\parbox{8cm}{By a \emph{normally bordered territory} of a slim rectangular lattice $L$ with a fixed \tbdia-diagram we mean a territory the sides of which are of normal slopes. }
\,\,\right\}
\label{eq:nwiRbzQgbV}
\end{equation}
\begin{equation}\left.
\parbox{8cm}{A \emph{rectangular} interval of $L$ 
is a non-chain interval that is a normally bordered territory as a region.}
\,\,\right\}
\label{eq:nrMBrdrhkVlmRg}
\end{equation}
So a rectangular interval is bordered by four $L$-compatible line segments of normal slopes and positive lengths, and it is also a region. (As opposed to normally bordered territories, a rectangular interval is always of a positive geometric area.)
For example, $L$  is a rectangular interval of itself. 
To see another example, note that   $[\Foot I, \Peak I]$ is not a rectangular interval of $\Shk 4$ in Figure \ref{fig6}  but, being isomorphic to $\Shk 2$, it is a rectangular lattice on its own right even though the subdiagram $[\Foot I, \Peak I]$ is not a \tbdia-diagram. 
\begin{equation}
\text{A \emph{rectangular $4$-cell} is a 4-cell that is a rectangular interval.}
\label{eq:cskPrnsfkRnsj}
\end{equation}
For a rectangular interval $I$ of $L$, the leftmost element and the rightmost element of $I$ are called the \emph{left corner} and the \emph{right corner} of $I$, and  they are denoted by
$\lcorner I$ and $\rcorner I$, respectively. Note that $\lcorner L$ and $\rcorner L$ are the two doubly irreducible elements of $L$.
For $x\in L$, we define the \emph{left support} and the \emph{right support} of $x$ by 
\begin{equation}
\lsupp x:=x\wedge \lcorner L\qquad\text{ and }\qquad 
\rsupp x:=x\wedge \rcorner L.
\label{eq:wKznPrsTsGb}
\end{equation}
Let us emphasize again that the concepts we defined in 
\eqref{eq:grnBndmlszGn},  \eqref{eq:hsGsmfKrnnSlk}, \eqref{eq:nwiRbzQgbV},  \eqref{eq:nrMBrdrhkVlmRg},  \eqref{eq:cskPrnsfkRnsj}, and \eqref{eq:wKznPrsTsGb} depend on the fixed \tbdia-diagram of $L$. 
We know from, say, the main result of Cz\'edli \cite{CzGpropmeet} that for an element $x$ of a slim rectangular lattice $L$,
\begin{equation}\left.
\parbox{10.3cm}{$\lsupp x\in\LBnd L$, $\rsupp x\in\RBnd L$, 
the interval $[\lsupp x,x]$ is a chain with all edges of normal slope $(1,1)$, and  the interval $[\rsupp x,x]$ is a chain with all edges of normal slope $(1,-1)$.}
\,\,\right\}
\label{eq:hwSukthhBf}
\end{equation}

Next, after some additional preparations, we recall the concept of \emph{multifork extensions} or, in other words, \emph{inserting a multifork}  from \cite{CzG:pExttCol}.  A 4-cell $H$ is called \emph{distributive} if the ideal $\ideal{\Peak H}$ is distributive. 
Note that a distributive 4-cell is rectangular; see \eqref{eq:nrMBrdrhkVlmRg}. At present, multifork extensions are only possible at distributive 4-cells. 
(In Section \ref{sect:thrust}, we will relax this assumption.)
The lattices $\Shk k$ for $k\in \Nplu:=\set{1,2,3,\dots}=\revised{\Nnul\setminus\set 0}$ were defined in \cite{CzG:pExttCol}; for $k\leq 4$ they are also given in Figure \ref{fig6}.
Let $H$ be a distributive 4-cell of a slim rectangular lattice $L$.
Before the formal description, note that Figure \ref{fig3} gives three examples. 
Namely,  for $i\in\set{1,2,3}$,  $L_{i}$ is obtained from $L_{i-1}$ by performing a $k_i$-fold multifork extension at $H_i$, where $H_i$ is indicated in the diagram of $L_{i-1}$, $k_1=2$, 
$k_2=1$, and $k_3=2$.

The formal definition runs as follows. Postponing the explanation of certain ingredients of Figure \ref{fig3}, let  $\Peak I$ and $\Foot I$  be the largest element of $\Shk k$ and the meet of the internal  (that is, non-boundary) meet-irreducible elements, respectively. In $L$, replace $H$ by a copy of $\Shk k$, that is, insert the elements of $[\Foot I, \Peak I]\setminus\set{\Peak I,\Foot I,\lcorner I,\rcorner I}$ as new lattice elements in the geometric interior of $H$. Usually, we need to resize and re-scale 
the diagram of $\Shk k$ so that the boundary of $\Shk k$ should be geometrically congruent to that of $H$; we know from the four sentences following (5.8) in  \cite{CzGdiagrectext} that such a resizing and rescaling are always possible. 
In the next step, for each new element $x$, draw a line segment in the direction $(-1,-1)$ down to $\LBnd L$ and a line segment  in the direction $(1,-1)$ down to $\RBnd L$. (Both line segments are of normal slopes.) Wherever a new line segment intersects an edge in an inner (geometric) point or terminates at an inner geometric point of an edge, we add a new lattice element. In this way, as we know from \cite{CzG:pExttCol}, we obtain a \tbdia-diagram of a new slim rectangular lattice $L'$. We say that $L'$ is obtained from $L$ by performing a \emph{$k$-fold multifork extension at the distributive $4$-cell $H$}. By an  \emph{$m$-by-$n$ grid} or (if we do not specify $m$ and $n$) a \emph{grid} we mean the direct product of an $(m+1)$-element chain and an $(n+1)$-element chain where $m,n\in\Nplu:=\set{1,2,3,\dots}$. Note that a grid cannot be a chain.
The importance of multifork extensions will be explained by 
Lemmas \ref{lemma:nmtfrkstZ} and \ref{lemma:mfXzsrjsmG}. The lemma below is 
\cite[Lemma 2.12 and Figure 6]{CzGlamps}, quoted from  \cite[Theorem 3.7]{CzG:pExttCol}.

\begin{lemma}[Multifork Sequence Lemma {\cite[Theorem 3.7]{CzG:pExttCol}}]
\label{lemma:nmtfrkstZ}
For each slim rectangular lattice $L$, there exist \revised{positive} integers $m_1,\dots, m_k$, a sequence  $L_0,L_1,\dots,L_k=L$ of slim rectangular lattices, and a  distributive $4$-cell 
$H_i$ of $L_{i-1}$ for $i\in\set{1,\dots,k}$ such that $L_0$ is a grid, $L_k=L$, and $L_i$ is obtained from $L_{i-1}$ by performing an $m_i$-fold multifork extension at $H_i$. Furthermore, any lattice obtained in this way from a grid is a slim rectangular lattice.
\end{lemma}

For a given $L$, the sequence in Lemma \ref{lemma:nmtfrkstZ} is not unique in general. However, let us agree that we always refer to a \emph{fixed list} when this lemma is used.
Furthermore, we refer to the system of $L_0,L_1,\dots,L_k=L$, 
$m_1,\dots, m_k$, and $H_1,\dots, H_k$  occurring in this lemma as the \emph{multifork sequence} of $L$.

Let $L$ be a slim rectangular lattice. Following \cite[Definition 2.3]{CzGlamps}, by a \emph{neon tube} we mean an edge $\nn=[p,q]$ such that $p\in\Mir L$. 
As before, $p=\Foot \nn$ and $q=\Peak \nn$. If $\Foot\nn\in\Bnd L$, then $\nn$ is a \emph{boundary neon tube}; otherwise  it is an \emph{internal neon tube}. The neon tubes in our figures are exactly the thick edges. Still going after  \cite[Definition 2.3]{CzGlamps},  \emph{lamps} are particular intervals. Namely, if $\nn$ is a boundary neon tube, then $\nn$ is a \emph{boundary lamp}. (However, we often say that $\nn$ is the neon tube of this lamp.) If $\pp$ is an internal neon tube, then it determines a unique \emph{internal lamp} $I$ by the rules $\Peak I:=\Peak\pp$ and $\Foot I:=\bigwedge\{\Foot \nn: \nn$ is an internal neon tube and $\Peak\nn=\Peak\pp\}$. We often say that the neon tubes $\nn$ occurring in this intersection are \emph{the neon tubes of} $I$. 
In our figures, the foot $\Foot I$ of a lamp is always \emph{black-filled} and it is mostly labeled by $I$ rather than $\Foot I$; this convention is explained by the fact that a lamp is uniquely determined by its foot; see \cite[Lemma 3.1]{CzGlamps}. For an internal lamp $I$, 
\begin{equation}\left.
\parbox{8cm}{the \emph{pegs} of $I$ are $\lcorner{\CircR I}$, $\rcorner{\CircR I}$, and the feet of the neon tubes of $I$;}
\,\,\right\}
\label{eq:tPgscsPsTz}
\end{equation}
a $k$-fold lamp has $k+2$ pegs. For example, in the diagram on the left of Figure \ref{fig17}, $J_1$ has five pegs and they are green-filled and pentagon-shaped. In the same diagram, $J_3$ has three  pegs; they are pentagon-shaped, the middle one is black-filled (since it is the foot of $J_3$), and the other two are green-filled.

Next, we define some polygons associated with intervals, in particular, with lamps. Let $I$ be an interval of $L$. By \eqref{eq:hwSukthhBf}, the intervals $[\lsupp{\Peak I}, \Peak I]$, $[\rsupp{\Peak I}, \Peak I]$, $[\lsupp{\Foot I}, \Foot I]$, and $[\rsupp{\Foot I}, \Foot I]$ are line segments of normal slopes; they are denoted by $\LRoof I$, $\RRoof I$, $\LFloor I$, and $\RFloor I$, respectively; see Figure \ref{fig3} for examples. The \emph{roof} and the \emph{floor} of $I$ are $\Roof I:=\LRoof I\cup\RRoof I$ and $\Floor I:=\LFloor I\cup \RFloor I$, respectively. 
If $T\in\set{\Floor I,\Roof I}$, then $T$ cuts $\FullRect L$ into two connected and topologically closed areas; the upper half and the lower half are the \emph{geometric filter} $\gfilter T$ and  the \emph{geometric ideal} $\gideal$ determined by $T$, respectively.
They consist of  geometric points. We note (but will not use) the fact that $\gideal{\Roof I}$ and $\gideal{\Floor I}$ are the regions determined by the lattice ideals (which are special intervals) $\ideal{\Peak I}$ and $\ideal{\Foot I}$, respectively. 
The \emph{illuminated set} of $I$ is
\begin{equation*}
\Enl I:= \gideal{\Roof I} \cup   \gfilter{\Floor I}.
\end{equation*}
The \emph{left illuminated set} $\LEnl I$ of $I$ consists of those (geometric) points of $\Enl I$ that are on the left of \revised{the rightmost neon tube}
of $I$. Similarly, the points of $\Enl I$ on the right of 
\revised{the leftmost neon tube of $I$}
constitute $\REnl I$, the  \emph{right illuminated set} of $I$. See Figure \ref{fig3} again for some examples.

Next, instead of being just an interval, let $I$ be a  lamp. 
If $I$ has exactly $k$ neon tubes, then it is \emph{$k$-fold lamp}.
Following \cite[Definition 2.6]{CzGlamps}, the \emph{body} of $I$, denoted by $\Body I$, is $I$ as a geometric region.   Note that $\Body I=\LEnl I\cap \REnl I$ and it is a territory. If $I$  has only one neon tube, then $\Body I$ is a line segment; otherwise it is a quadrangle with precipitous upper sides while its lower sides are of normal slopes. Still going after \cite[Definition 2.6]{CzGlamps}, the \emph{circumscribed rectangle} $\CircR I$ of an internal lamp is (the geometric region determined by) the interval $[p,\Peak I]$ where $p$ is the meet of the pegs of $I$, see \eqref{eq:tPgscsPsTz} (equivalently, $p$ is the meet of the leftmost lover cover of $\Peak I$ and the rightmost lower cover of $\Peak I$). 
It follows easily from \eqref{eq:hwSukthhBf} that, for an interval $I$ and an \emph{internal} lamp $K$ of a slim rectangular lattice $L$,
\begin{equation}\left.
\parbox{9.9cm}{$\LEnl I$,  $\REnl I$, and $\Body K$ are territories while  $\gideal{\Roof I}$, $\gfilter{\Roof I}$, $\gideal{\Floor I}$, $\gfilter{\Floor I}$, $\Enl I$,  and $\CircR K$ are normally bordered  territories; see \eqref{eq:hsGsmfKrnnSlk} and  \eqref{eq:nwiRbzQgbV}.}
\,\,\right\}
\label{eq:tRfjkpNg}
\end{equation}

A lamp $I$ and the pair $(\Foot I,\Peak I)$ mutually determine each other; this allows us to call $(\Foot I,\Peak I)$ a ``lamp'' (in quotient marks) in the following lemma. From Lemma \ref{lemma:nmtfrkstZ} (here) and \cite[(2.10)]{CzGlamps}, we obtain the following statement.

\begin{lemma} [Each internal lamp comes to existence by a multifork extension]\label{lemma:mfXzsrjsmG}
For the fixed multifork sequence of $L$ in Lemma \ref{lemma:nmtfrkstZ}, the set of internal lamps of $L$ is of the form $\set{I_j:1\leq j\leq k}$ where, for $j\in\set{1,\dots,k}$, the ``lamp" $(\Foot{I_j}, \Peak{I_j})$ comes to existence by the $j$-th multifork extension,  $\CircR {I_j}$ in $L=L_k$ is the geometric region determined by $H_j$ in $L_{j-1}$, and $\Foot{I_{j}}\in L_j\setminus L_{j-1}$. 
\end{lemma}

With the notations of Lemma \ref{lemma:mfXzsrjsmG}, the interval $I_j=[\Foot{I_j}, \Peak{I_j}]$ of $L$ can have more elements than the interval $[\Foot{I_j}, \Peak{I_j}]$ of $L_j$. (For an example, see $J_1\in\Lamp L$ on the left of Figure \ref{fig17} but do not confuse the subscript of $J_1$ with $j$.)
If $1\leq i<j\leq k$, then we say that the lamp $I_j$ is \emph{younger} than $I_i$.

For a subset $X$ of the plane, the \emph{topological interior} of $X$ will be denoted by $\TopInt X$. We borrow the following definition from \cite[Definition 2.9]{CzGlamps}.

\begin{definition}\label{def:rellamps} For a slim rectangular lattice $L$ and lamps $I,J$ of $L$, let $(I,J)\in \rhfoot$ mean that $I\neq J$, $I$ is an internal lamp, and $\Foot I\in \Enl J$. 
Similarly, let $(I,J)\in \rhinfoot$  mean that $\Foot I\in \TopInt{\Enl J}$. Also, let  $(I,J)\in \rhlrbody$ mean that  $I\neq J$,  $I$ is an internal lamp, and $\Body I\subseteq \LEnl J$ or $\Body I\subseteq \REnl J$. 
Finally,  $(I,J)\in \rhbody$ means that  $I\neq J$,  $I$ is an internal lamp, and $\Body I\subseteq \Enl J$. Letting ``$\leq$'' be the reflexive transitive closure of $\rhinfoot$, we denote by $\Lamp L=(\Lamp L;\leq)$ the poset of lamps of $L$.
\end{definition}

Note that  $(I,J)\in\rhinfoot$  implies that $I\neq J$ and $I$ is an internal lamp. The congruence generated by a pair $(x,y)$ of elements will be denoted by $\con(x,y)$. 
The following statement recalls a part of  \cite[Lemma 2.11]{CzGlamps}.

\begin{lemma}[\cite{CzGlamps}]\label{lemma:vltfLm} If $L$ is a slim rectangular lattice, then $\rhfoot$ $=$ $\rhinfoot$ $=$ $\rhlrbody$ $=$ $\rhbody$, $\Lamp L=(\Lamp L;\leq)$ is indeed a poset, and
whenever if $I\prec J$ in $\Lamp L$, then $(I,J)\in\rhinfoot$. Furthermore,  $(\Lamp L;\leq) \cong (\Jir{\Con L};\leq)$ and the map
\begin{equation}
\text{$\phi\colon (\Lamp L;\leq) \to (\Jir{\Con L};\leq)$ defined by $I\mapsto \con(\Foot I,\Peak I)$}
\label{eq:sznZtcsTl}
\end{equation} 
is an order isomorphism.
\end{lemma}

\section{Quotient diagrams}\label{sect:diagrquotient}
This section is devoted to the proof of Theorem \ref{thm:mrKtpnMgsRflpbsmG}.
By an \emph{edge segment} we mean a \emph{geometric} line segment of \emph{positive length} that is a part of an edge of the diagram. In other words, any two distinct geometric points of an edge of the fixed \tbdia-diagram determine an edge segment.  We are going to prove two easy lemmas.

\begin{lemma}\label{lemma:hmKntzSth} 
Let $L$ be a slim rectangular lattice. If $H$ is a rectangular interval of $L$, see \eqref{eq:nrMBrdrhkVlmRg}, $J$ is an internal lamp of $L$, and an edge segment of 
a neon tube of $J$ lies in  $H$, then $\Body J\subseteq H$.
\end{lemma}

\begin{proof}
Let $\pp$ be a neon tube of $J$ such that an edge segment of $\pp$ lies in $H$. Since $\pp$ is precipitous and $H$ is normally bordered, planarity implies that  $\pp$ and, in particular, $\Peak J=\Peak \pp$ lie in $H$. If $\pp$ is the only neon tube of $J$, then $\Body J$ is the edge $\pp$, whence $\Body J\subseteq H$ is clear. So assume that $\pp$ is not the only neon tube of $J$. Since $H$ is normally bordered, an (upper) edge segment of each neon tube $\qq$ of $J$  lies within $H$, whence $\qq$  fully lies in $H$. In particular, 
$\Foot \qq$ lies in $H$.
Let $\qq_1$ and $\qq_2$ be the leftmost neon tube and the rightmost neon tube of $J$, respectively. Then $\Foot J=\Foot{\qq_1}\wedge \Foot{\qq_2}$. Since $\Foot{\qq_1}$ and $\Foot{\qq_2}$ (as distinct lower covers of $\Peak J$) are incomparable,
\revised{the main result of \cite{CzGpropmeet}\semmi{used already at \eqref{eq:hwSukthhBf}}}
implies that 
$[\Foot J,\Foot{\qq_1}]$ and $[\Foot J,\Foot{\qq_2}]$
are chains of normal slopes. Thus,  using that $H$ is normally bordered and both  $\Foot{\qq_1}$ and  $\Foot{\qq_2}$ lie in $H$, we obtain that  $\Foot J$ lies in $H$. The two lower sides of $\Body J$ are of normal slopes while its upper edges are precipitous.  Therefore, since $H$ is normally bordered and both $\Peak J$ and $\Foot J$ lie in $H$, we conclude 
\revised{the required} $\Body J\subseteq H$. 
\end{proof}

As usual, if $\pp$ and $\qq$ are prime intervals (that is, edges) of a slim semimodular lattice $K$ such that  $\Foot\pp\wedge \Peak \qq= \Foot\qq$ and $\Foot\pp\vee \Peak \qq= \Peak\pp$, then $\pp$ is  \emph{down-perspective to $\qq$} and $\qq$ is \emph{up-perspective} to $\pp$. If $\pp$ is up-perspective or down-perspective to $\qq$, then we say that $\pp$ and $\qq$ are \emph{perspective}. Let $\tau$ denote the least equivalence relation on the set of the prime intervals of $K$ that includes the perspectivity relation. Following Cz\'edli and Schmidt \cite{CzGSch-jordanh}, the $\tau$-block of $\pp$ is called the \emph{trajectory through}  $\pp$. The 4-cells of a trajectory are those 4-cells that are formed by two consecutive edges of the trajectory. Each trajectory has a unique \emph{top edge}, that is, an edge to which any other edge of the trajectory is up-perspective; this top edge is a neon tube. Furthermore, the top edge is the only neon tube of a trajectory.
Therefore, we will also call the top edge the \emph{neon tube of the trajectory}.

For a slim rectangular lattice $L$, 
let $\pp$ be a neon tube of an  $I\in\Lamp L$. Denote by  $i$  the smallest subscript such that $L_i$ from the (fixed) multifork sequence of  $L$, see Lemma \ref{lemma:nmtfrkstZ}, contains $\Foot I$ (equivalently, it contains $\pp$). Motivated by \cite{CzGdiagrectext}, 
\begin{equation}\left.
\parbox{10 cm}{the \emph{original territory} $\OT\pp$ of $\pp$ is the union of the geometric regions determined by the 4-cells of the trajectory containing $\pp$ in $L_i$ while the \emph{lower half of the original territory} of $\pp$, denoted by $\LHOT\pp$, is the union of the geometric regions of those 4-cells of the trajectory in $L_i$ that do not contain $\Peak \pp$.}\,\,\right\}
\label{eq:pXwrnkrgPCxsk}
\end{equation} 
Note that in $L$, the above-mentioned 4-cells can be divided into smaller 4-cells by edges born after $I$. For example, with $(I,\pp,L):=\revised{(J_1,\pp',L')}$ in Figure \ref{fig17}, the original territory of $\pp$ is ``\spiralfil-filled''  and $\LHOT\pp$ is both curl-filled and yellow-filled. (This multi-purpose figure  comes later.) 
For  a lamp $J\in \Lamp L$, we define  
\begin{equation}
\text{$\UHCircR J$, the \emph{upper half of the circumscribed rectangle of $J$}}
\end{equation} 
as the union of those $4$-cells lying in $\CircR J$ that have the same peak as $J$. E.g., on the left of Figure \ref{fig17},  
$\UHCircR{J_1}$ and $\UHCircR{J_3}$ are flower-filled (with magenta background) while $\UHCircR{J_4}$ and $\UHCircR{J_5}$ are wave-filled. Observe that 
\begin{equation}\left.
\parbox{10.7cm}{the $L$-compatible line segments forming the border of $\UHCircR J$ are edges of normal slopes and,  except for the edge segments lying on \revised{the neon tubes of $J$}, the geometric interior
$\TopInt{\UHCircR J}$ of $\UHCircR J$ contains neither an element of $L$, nor an edge segment.}\,\,\right\}
\label{eq:ncghnvmgnBZcfDs}
\end{equation}
Before verifying this, note that it follows from  Cz\'edli and Schmidt \cite[Lemma 15]{CzgScht-a-visual} or from Lemma \ref{lemma:nmtfrkstZ} (here) that 
\begin{equation}\left.
\parbox{6cm}{an ideal of $L$ is distributive if and only if no precipitous edge segment lies in it.}\,\,\right\}
\label{eq:mlHtntbkNcZcs}
\end{equation}
Next, observe that \eqref{eq:ncghnvmgnBZcfDs} held when $J$ had just been born, see Lemmas \ref{lemma:nmtfrkstZ} and \ref{lemma:mfXzsrjsmG}. 
Subsequent multifork extensions cannot destroy the validity of 
 \eqref{eq:ncghnvmgnBZcfDs}  since they are performed at \emph{distributive} 4-cells, say, $H_i$ and  \eqref{eq:mlHtntbkNcZcs} applies to the ideals $\ideal{\Peak{H_i}}$.
We have proved \eqref{eq:ncghnvmgnBZcfDs}. For later reference, note that \eqref{eq:mlHtntbkNcZcs} trivially implies that
\begin{equation}\left.
\parbox{6cm}{$\LHOT \pp$ from \eqref{eq:pXwrnkrgPCxsk} is the union of two disjoint normally bordered territories;}\,\,\right\}
\label{eq:znksnkhLpktTj}
\end{equation}
it may happen that one or two of these territories is of geometric area 0.

\begin{lemma}\label{lemma:hBbrSzspnrJt} 
Let $L$ be a slim rectangular lattice and  $I,J\in\Lamp L$ such that $J$ is an internal lamp and $J\neq I$. If an edge segment of a neon tube of  $J$ lies in  $\Enl I$, then $\Body J\subseteq I$ and, in $\Lamp L$, $J<I$.  Similarly, if an edge segment of a neon tube of  $J$ lies in the original territory $\OT\pp$ of a neon tube $\pp$ of $I$, then $\Body J$ also lies in $\OT\pp$ and even in $\LHOT\pp$; see \eqref{eq:pXwrnkrgPCxsk}.
\end{lemma}

\begin{proof}
Using the notations of Lemmas \ref{lemma:nmtfrkstZ} and \ref{lemma:mfXzsrjsmG} and letting $I_i:=I$ and $I_j:=J$, we know from \eqref{eq:mlHtntbkNcZcs} that no precipitous edge segment lies in the ideal $\ideal{\Peak{H_i}}$ of $L_{i-1}$.
Hence, $J$ is younger than $I$, that is, $i<j$. There are  two kinds of 4-cells of $L_i$ that lie in $\Enl {I_i}$, understood in $L_i$. First, there are $m_i+1$ non-rectangular 4-cells $C_0$,\dots, $C_{m_i}$ (where $m_i$ is the number of neon tubes of $I=I_i$). Second, there can be some rectangular 4-cells $D_1$,\dots, $D_t$ in $\Enl{I_i}$.   
After the $i$-th multifork extension, which created $I=I_i$, later multifork extensions can cut some of the rectangular 4-cells $D_1$,\dots, $D_t$ into smaller regions, so all we can say that (keeping the earlier notation) $D_1$, \dots, $D_t$ are rectangular intervals of $L$. However, it follows from \eqref{eq:mlHtntbkNcZcs} that no multifork extension cuts the 
the 4-cells $C_0$,\dots, $C_{m_i}$ into smaller regions. Hence, 
\begin{equation}\left.
\parbox{10cm}{$\Enl I$ is \emph{geometrically} partitioned into non-rectangular 4-cells with peaks equal to $\Peak I$ and rectangular intervals with peaks different from $\Peak I$;}
\,\,\right\}
\label{eq:blBsGgzkRdKlr}
\end{equation}
here ``geometrically partitioned'' means that the geometric area of the intersection of any two of the mentioned geometric objects (4-cells and rectangular intervals) is $\emptyset$ or a line segment.
It follows from \eqref{eq:blBsGgzkRdKlr} and $J\neq I$ that an edge segment of a neon tube of $J$ lies in one of the rectangular intervals  $D_1$, \dots, $D_t$. Thus, by Lemma \ref{lemma:hmKntzSth},  $\Body J$ lies entirely in $\Enl I$, as required. Finally, since $(J,I)\in\rhbody$,  Lemma \ref{lemma:vltfLm} yields that $J<I$, as required.  
\revised{This proves the first half of the lemma.} 

\revised{The second half follows in the same way; the only difference is that instead of $m_i+1$,  now $\Enl\pp$ in $L_i$ has one  (if $\pp$ is a boundary neon tube)  or two non-rectangular 4-cells. The proof of Lemma \ref{lemma:hBbrSzspnrJt} is complete.}
\end{proof}

\begin{definition}Let $I$ be a lamp of a slim rectangular lattice $L$.  For a lattice element $x\in L\setminus\set{\Peak I}$ lying on $\LRoof I$ but not on $\LFloor I$, the interval $[x\wedge \Foot I, x]$ is a \emph{left rung} of $I$.
Similarly, the \emph{right rungs} of $I$ are the  intervals of the form $[x\wedge \Foot I, x]$ where $x\in L\setminus\set{\Peak I}$ lies on $\RRoof  I$ but not on $\RFloor I$. The \emph{rungs} of $I$ are its left rungs and right rungs.
\end{definition}

For example, in Figure \ref{fig3}, the rungs of $I_1$ in $L_1$ 
and those of $I_2$ in $L_3$ are indicated by dotted ovals;  $I_2$ in $L_3$ has three left rungs and one right rung. This figure together with the following lemma indicate why we name some intervals associated with the lamp $I$ after the short bars forming the steps of a ladder.

\begin{lemma}\label{lemma:RnlxtG}
 The rungs of a lamp $I$ are non-singleton chain intervals of  normal slopes; the left rungs are of slope $(1,-1)$ while the right rungs are of slope $(1,1)$. Furthermore, $I$ has a left rung if and only if $I$ is not a left boundary lamp. Similarly, $I$ has a right rung if and only if $I$ is not a right boundary lamp.
\end{lemma}

\begin{proof} The statement follows from 
\revised{the main result of \cite{CzGpropmeet}}\semmi{\eqref{eq:hwSukthhBf}}
 and Lemma \ref{lemma:mfXzsrjsmG}.
\end{proof}

Next, for a lamp $I$ of $L$, let $\rho(I)$ be the equivalence of (the underlying set of) $L$ whose nonsingleton blocks are $I$ and the rungs of $I$. Figure \ref{fig3} with $L=L_3$ and $I=I_2$ witness that $\rho(I)$ need not be a congruence of $L$. However, we have the following lemma.

\begin{lemma}\label{lemma:RnGcN} 
If $I$ is \revised{a \emph{minimal}} lamp of a slim rectangular lattice $L$, 
then $\rho(I)$ defined above is a congruence of $L$ and, furthermore, $\rho(I)=\phi(I)$; see \eqref{eq:sznZtcsTl}.
\end{lemma}

\begin{figure}[ht] \centerline{ \includegraphics[scale=1.0]{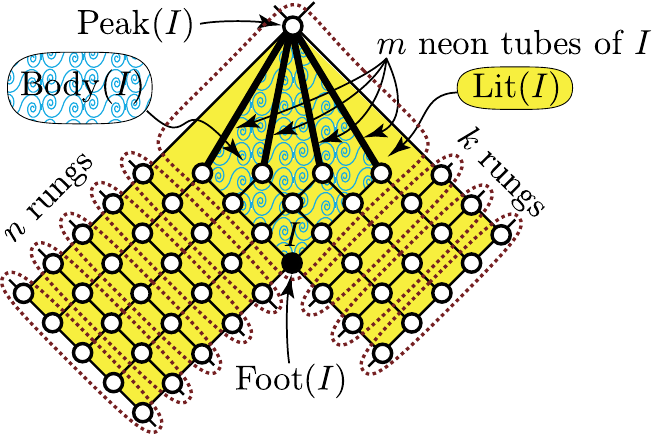}} \caption{Illustrating the proof of   Lemma \ref{lemma:RnGcN} for $(m,n,k)=(4,5,3)$}\label{fig-16}
\end{figure}

\begin{proof} 
First, let us assume that $I$ is an internal lamp.
By Lemmas \ref{lemma:vltfLm} and \ref{lemma:hBbrSzspnrJt}, the minimality of $I$ in $\Lamp L$ \revised{implies} that no precipitous edge segment lies in $\Enl I$. 

Let $x\in L$ lie in $(\Enl I\setminus\Body I)\setminus\Roof I$. Then $x$ is not on the upper boundary of $L$. Furthermore, $x$ is not the foot of an internal (i.e.,  precipitous) neon tube since no precipitous edge segment lies in $\Enl I$. Therefore, since we are in a \tbdia-diagram, $x$ is not an internal meet-irreducible element of $L$. 
It follows from Gr\"atzer and Knapp \cite[Lemma 3]{GKnapp-III}, see also Cz\'edli and Schmidt \cite[(2.14)]{CzGSchT--patchwork} for a more explicit source, that
any meet-irreducible boundary element of $L$ lies on the upper boundary of $L$. Hence, $x$ is meet-reducible.
This fact together with either semimodularity and \eqref{eq:mnHhwlCh} or Gr\"atzer \cite[Lemma 8]{GKn-I}  imply that $x$ has exactly two covers. 
By left-right symmetry, we can assume that $x$ lies in $\LEnl I$. Let $x_0:=x$. While $x_i$ still lies in  $(\Enl I\setminus\Body I)\setminus\Roof I$ (and then it has exactly two covers by the same reason as $x$), let $x_{i+1}$ be the left upper cover of $x_i$ and observe that $x_{i+1}$ lies in $\Enl I\setminus\Body I$. In this way, we obtain a chain $[x,y]=[x_0,x_t]$ such that $y\in\LRoof I$.
In the absence of precipitous edge segments lying in $\Enl I$, this chain is an interval of normal slope $(1,-1)$. But the left rung starting at $y$ and going to the southeast is of the same slope, whereby this rung contains $x$.  
Therefore, every lattice element lying in $\Enl I\setminus\Body I$ belongs to a  (unique) rung of $I$.

Let $m$, $n$, and $k$ denote the number of neon tubes of $I$, that of the left rungs of $I$, and that of the right rungs of $I$, respectively; see Figure \ref{fig-16} for an illustration. 
As this figure shows, $\Enl I$ is obtained by gluing an $(n-1)$-by-$m$ grid to the lower left boundary (apart from its bottom 0) of $\Shk m$ and a
$(k-1)$-by-$m$ grid to the lower right boundary (apart from 0) of $\Shk m$.
 Let $\pp$ be an arbitrary edge of a rung of $I$ or an edge lying in $\Body I$ (that is, an edge of $I$), and let $\con(\pp)\in\Con L$ denote the congruence generated by $\pp$. It follows trivially from the (lattice) structure of $\Enl I$, see Figure \ref{fig-16} and the Swing Lemma, which is due to Gr\"atzer \cite{gG15} (see also Cz\'edli, Gr\"atzer, and Lakser \cite{CzGGLakser} or Cz\'edli and Makay \cite{CzGMakay} for alternative approaches) that 
\begin{equation}\text{$\con(p)$ collapses an edge of $L$ if and only if so does $\rho(I)$.}
\label{eq:swzsznnKbc}
\end{equation}
Note that, in a slightly less convenient way, \cite[Theorem 4.4]{CzG:pExttCol} also implies \eqref{eq:swzsznnKbc}. 

Finally, \eqref{eq:swzsznnKbc}  and the fact that both the $\con(p)$-blocks and the $\rho(I)$-blocks are intervals imply that $\rho(I)=\con(p)=\phi(I)$, completing the case when $I$ is an internal lamp. 
The case when $I$ is a boundary lamp follows trivially from Lemma \ref{lemma:nmtfrkstZ} and the Swing Lemma. 
This completes the proof of Lemma \ref{lemma:RnGcN}
\end{proof}

Now we are in the position to prove Theorem \ref{thm:mrKtpnMgsRflpbsmG}.

\begin{figure}[ht] \centerline{ \includegraphics[width=0.99\textwidth]
{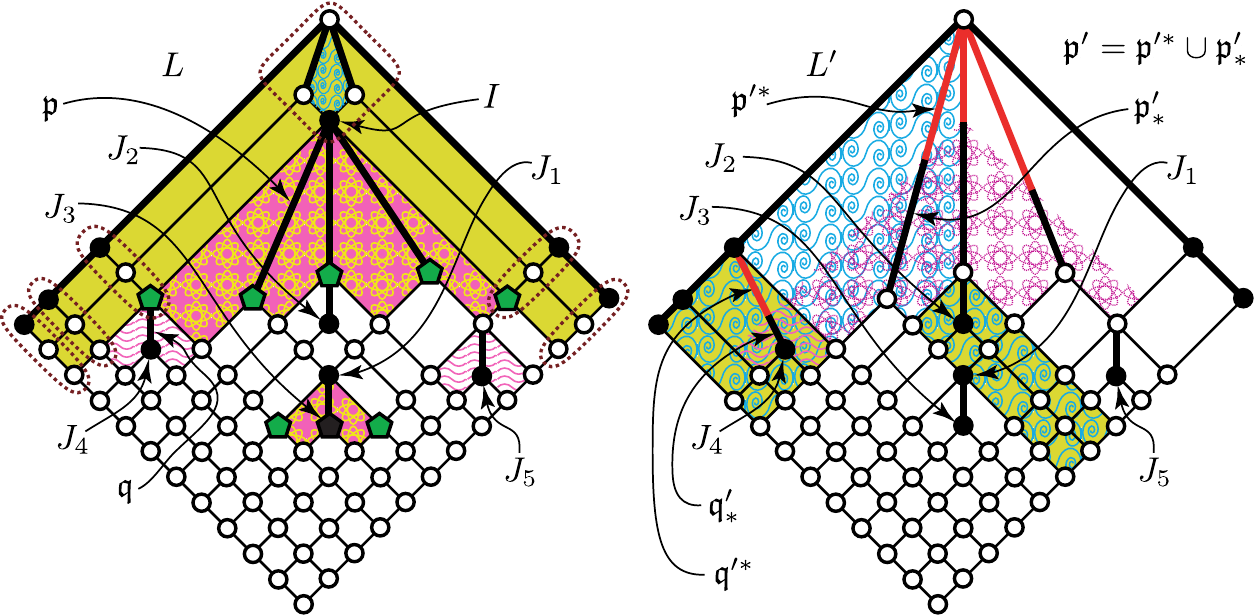}} \caption{$L$ and $L'=L/\alpha$ with $\alpha=\phi(I)$ and internal $I\in\Lamp L$}\label{fig17}
\end{figure}

\begin{figure}[ht] \centerline{ \includegraphics[scale=1.0]{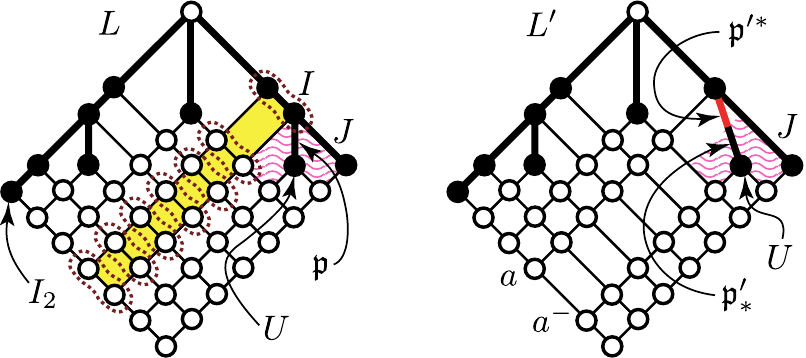}} \caption{$L$ and $L'=L/\alpha$ with $\alpha=\phi(I)$ and boundary $I\in\Lamp L$}\label{fig-case1}
\end{figure}

\begin{proof}[Proof of  Theorem \ref{thm:mrKtpnMgsRflpbsmG}]
Observe the following. Let $\qq$ be an edge,  let $\vec u$ be a vector upwards with a normal or a precipitous slope, and let $\qq'$ be the line segment obtained from $\qq$ so that we move $\Peak \qq$ by $\vec u$ while keeping $\Foot \qq$ unchanged. It is a trivial geometric fact that  
\begin{equation} \text{if $\qq$ is precipitous, then so is $\qq'$.}
\label{eq:mkrThvRsl}
\end{equation}
Vectors of the form $(1,q)$ with $q\in\RR$ such that $|q|<1$ and lines parallel to such vectors are said to be of a \emph{slight slope}. In other words, a line (segment) is of a slight slope if and only if it is neither precipitous nor of a normal slope. 
Based on the concept given in \eqref{eq:tPgscsPsTz}, we need the following (easy geometric) consequence of \eqref{eq:ncghnvmgnBZcfDs}. 
\begin{equation}\left.
\parbox{8cm}{The line through two different pegs of an internal lamp is of a slight slope.}
\,\,\right\}
\label{eq:chVzkRkpmRk}
\end{equation}

First, we only deal with the particular case of Theorem \ref{thm:mrKtpnMgsRflpbsmG} when $\alpha$ is an atom of $\Con L$. Then $\alpha\in\Jir{\Con L}$, so we obtain by Lemma \ref{lemma:vltfLm} that $\alpha$ is of the form $\phi(I)$ for a minimal lamp $I$ of $L$. Hence, Lemmas \ref{lemma:RnlxtG} and \ref{lemma:RnGcN} apply for the non-singleton $\alpha$-blocks: they are $I$ and the rungs of $I$, and these rungs are of normal slopes. 

Let $L'=L/\alpha$, see Figures \ref{fig17} and \ref{fig-case1} (but mainly Figure \ref{fig17}), and let $D'$ be the diagram we obtain from the fixed \tbdia-diagram of $L$ by removing all elements but the largest one from every block of $\alpha$. That is, we remove all lattice elements that lie in $\Enl I\setminus\Roof I$. We need to show that $D'$ is a \tbdia-diagram of $L'$.  (In the two figures just mentioned, $D'$ is \revised{the diagram} on the right.)

The edges of (the fixed \tbdia-diagram of) $L$ fall into three categories; to define \revised{these categories}, let $\pp$ be an edge of $L$.
\begin{itemize}
\item If both $\Foot \pp$ and $\Peak \pp$ are the largest elements in their $\alpha$-blocks (in particular, if these blocks are singletons), then $\pp$ is a \emph{permanent edge} and we let $\pp':=\pp$, which is an edge of $D'$.
\item If $\Foot \pp$ is the largest element of its $\alpha$-block but $\Peak\pp$ is not, then $\pp$ is a \emph{changing edge} and we define the edge $\pp'$ of $D'$ by letting $\Peak {\pp'}$ be the largest element of the $\alpha$-block $\Peak\pp/\alpha$ and $\Foot{\pp'}:=\Foot\pp$.
\item If $\Foot \pp$ is not the largest element of its $\alpha$-block, then $\pp$ is a \emph{vanishing edge} and $\pp'$ is undefined. 
\end{itemize}
The edges of $D'$ are the $\pp'$ for changing or permanent edges $\pp$ of $L$. Since planarity excludes that an edge of $L$ ``jumps over'' $\Floor I$, we have that
\begin{equation}\left.
\parbox{9cm}{an edge $\pp$ of $L$ is a changing edge $\iff$ $\Peak\pp$ lies on $\Floor I$ and $\Foot\pp$ lies in $\gideal{\Floor I}$ but not on $\Floor I$.}
\,\,\right\}
\label{eq:mstrbRdnKcsbp}
\end{equation}
For a changing edge $\pp$,   $\pp'$ is geometrically  partitioned into two edge segments,
\begin{equation}
\lo{\pp'}:=\pp'\cap\gideal{\Floor I}\quad\text{ and }\quad
\hi{\pp'}:=\pp'\setminus\lo{\pp'}.
\label{eq:ksnKmrhgkSLt}
\end{equation}
For some changing edges, including $\pp$ and $\qq$,  of $L'$, Figures \ref{fig17} and \ref{fig-case1} show the concepts defined in \eqref{eq:ksnKmrhgkSLt}. Namely, $\hi{\pp'}$ is drawn in red while $\lo{\pp'}$ remains black. For $K\in\Lamp L$ having a changing neon tube, the fill pattern of $\UHCircR K$ is repeated in $L'$ in the two figures. Note that \eqref{eq:ksnKmrhgkSLt} is  meaningful for a permanent edge $\pp$, too, but then $\lo{\pp'}=\pp$ and $\hi{\pp'}=\emptyset$.
Next, we claim that 
\begin{equation}\left.
\parbox{10.5cm}{for any non-vanishing edge $\pp$ of $L$, if $\pp$ is of a normal slope, then $\pp'$ is of the \emph{same} normal slope. Also, if $\pp$ is precipitous, then so is  $\pp'$.}\,\,\right\}
\label{eq:gzlSmskLdlbRdlth}
\end{equation}
It suffices to deal with a changing edge $\pp$.  The first case is when  $\Peak\pp$ lies on $\LFloor I\setminus\set{\Foot I}$. Note that the rung $r$ with foot $\Peak\pp$ is of slope $(1,-1)$ by Lemma \ref{lemma:RnlxtG} and 
$r=\Peak\pp/\alpha$ by Lemma \ref{lemma:RnGcN}. By \eqref{eq:mstrbRdnKcsbp}, $\pp$ does not lie on $\LFloor I$, whereby $\pp$ is not of slope $(1,1)$.
If $\pp$ is of slope $(1,-1)$, then we move $\Peak \pp$ along $r$ of the same slope to obtain $\pp'$, whereby $\pp$ and $\pp'$ are of  the same slope,  as required. If $\pp$ is precipitous, then so is $\pp'$ by \eqref{eq:mkrThvRsl},  as required.  The case when $\Peak\pp$ lies on $\RFloor I\setminus\set{\Foot I}$ is analogous.

By  \eqref{eq:mstrbRdnKcsbp}, we are left with the case where  $\Peak\pp=\Foot I$. 
If $I$ is an internal lamp, then it follows easily from Lemmas \ref{lemma:nmtfrkstZ} and \ref{lemma:mfXzsrjsmG}, or simply from the fact that $\Foot I$ lies in $\TopInt{\CircR I}$, that 
the vector $\vec w$ from $\Foot I$ to $\Peak I$ is precipitous. Since $\Floor I$ is $\Lambda$-shaped,  \eqref{eq:mstrbRdnKcsbp} gives that our changing edge $\pp$ is also precipitous. Hence, \eqref{eq:mkrThvRsl} yields that  $\pp'$ is precipitous, as required. 
If $I$ is a boundary lamp, say, a right boundary lamp as in Figure \ref{fig-case1}, then $\vec w$ is of (normal) slope $(-1,1)$. Then \eqref{eq:mstrbRdnKcsbp} excludes that our changing edge $\pp$ is of slope $(1,1)$. So either $\pp$ is of the same (normal) slope $(-1,1)$ as $\vec w$ and $\pp'$ is also of this slope, or $\pp$ is precipitous and \eqref{eq:mkrThvRsl} implies that $\pp'$ is also precipitous. 
We have shown \eqref{eq:gzlSmskLdlbRdlth}.

Next, inspecting its edges, we are going to show that $D'$ is a \tbdia-diagram and  it is a \tbdia-diagram of $L/\alpha$. Note that if no two edges of $D'$ overlap or cross each other and no edge of $D'$ goes through a vertex of $D'$ (different from the endpoints of the edge in question), then $D'$ is a planar diagram and it is the Hasse diagram of $L/\alpha$.

First, we show that for every edge $\pp$ of $L$, the corresponding new edge 
\begin{equation}
\text{$\pp'$ of $D'$ does not go through any vertex of $D'$.}
\label{eq:swkHrmM}
\end{equation}
This is clear if $\pp$ is of a normal slope since then $\lo{\pp'}=\pp$  and  $\hi{\pp'}$ lies in $\Enl I$ but, apart from $\Roof I$, $\Enl I$ contains no vertex of $D'$. Assume that $\pp$ is precipitous; so is $\pp'$ by  \eqref{eq:mstrbRdnKcsbp}.
Since $\Enl I\setminus \Floor I$ contains no vertex of $D'$, 
it is only  $\lo{\pp'}$ that could go through a vertex different from $\Foot{\pp '}=\Foot\pp$. But \eqref{eq:chVzkRkpmRk} excludes this possibility since $\pp'$ is precipitous. We have shown \eqref{eq:swkHrmM}. \eqref{eq:swkHrmM} implies that

\begin{equation}\left.
\parbox{7cm}{no overlapping occurs among the edges of $D'$, whereby $D'$ is a Hasse-diagram of $L/\alpha$.}\,\,\right\}
\label{eq:mhtlPlgBrkdmDj}
\end{equation}

Next, to show the planarity of $D'$, we have to exclude that two edges, \revised{say} $\pp$ and $\qq$, of $D'$ intersect at a geometric point that is not a vertex of $D'$. In view of \eqref{eq:mhtlPlgBrkdmDj}, we can assume that $|\set{\Foot \pp,\Peak\pp,\Foot\qq,\Peak\qq}|=4$.
There are three cases; they are easy and, after the first case, some trivial details will be omitted.

\begin{case}
Let $\pp$ and $\qq$ be precipitous changing edges. By the ``$=4$ assumption'',
they are neon tubes of different lamps $U$ and $V$, respectively. These two lamps are internal since $\pp$ and $\qq$ are precipitous. 
By definition,  an internal lamp is determined by its peak. Hence $\Peak \pp=\Peak U\neq\Peak V=\Peak \qq$.
If $\Peak U$ and  $\Peak V$ lie on $\LFloor I$ and, say, $\Peak U>\Peak V$, then the line through $\RRoof V$ separates not only $\pp$ from $\qq$ but also $\pp'$ from $\qq'$; see Figure \ref{fig17} with $(U,V):=(J_1,J_4)$. Hence, $\pp'$ and $\qq'$ do not cross each other. 
The situation when $\Peak U$ and $\Peak V$ lie on $\RFloor I$ is analogous.

So assume that $\Peak U=\Peak\pp$ lies on $\LFloor I\setminus \RFloor I$  and  $\Peak V=\Peak\qq$ lies on $\RFloor I\setminus\LFloor I$. Then the upper left edge of $\UHCircR U$, which is the upper left side of $\CircR U$ and it is of normal slope by \eqref{eq:ncghnvmgnBZcfDs}, lies on $\LFloor I$ but does not contain $\Foot I$. Similarly, the upper right edge of $\UHCircR V$ lies on $\RFloor I\setminus\set{\Foot I}$. By planarity and \eqref{eq:ncghnvmgnBZcfDs}, 
\begin{equation}
|\UHCircR U\cap \UHCircR V|\subseteq\set{\rcorner{\CircR U}}\cap \set{\lcorner{\CircR V}} .
\label{eq:szCstkbzVlg}
\end{equation}
In particular,  the intersection of these two geometric areas contains at most one geometric point.   \eqref{eq:mkrThvRsl} and \eqref{eq:chVzkRkpmRk} imply that 
\begin{equation}\left.
\parbox{8cm}{for the neon tube $\pp$ of our internal lamp $U$, $\lo{\pp'}$ lies in $\UHCircR U\setminus \set{\lcorner{\CircR U}, \rcorner{\CircR U}}$.}\,\,\right\}
\label{eq:nZnkFrBgMjs}
\end{equation}
Since the same holds for $(\qq,V)$ in place of $(\pp,U)$, \eqref{eq:szCstkbzVlg} implies that $\lo{\pp'}$ and $\lo{\qq'}$ do not intersect. Neither do $\hi{\pp'}$ and $\hi{\qq'}$ since $\hi{\pp'}$ lies in $\LEnl I\setminus\Body I$ while 
$\hi{\qq'}$ lies in $\REnl I\setminus\Body I$. Therefore, $\pp'$ and $\qq'$ do not cross each other, as required. 
\end{case}

\begin{case} Assume that $\pp$ is a precipitous changing edge and $\qq$ is a non-vanishing edge of a normal slope. Then \eqref{eq:gzlSmskLdlbRdlth} and the fact that the vanishing edges have been removed imply in a straightforward way that $\pp'$ and $\qq'$ do not cross each other. 
\end{case}

\begin{case} If $\pp$ and $\qq$ are non-vanishing edges of normal slopes, then $\pp=\lo{\pp'}$ and $\qq=\lo{\qq'}$ imply easily that $\pp'$ and $\qq'$ do not intersect each other. 
\end{case}

Now that the possible cases have been considered, $D'$ is indeed a planar diagram of $L'$. \revised{It follows easily from the minimality of $I$ and Lemmas \ref{lemma:nmtfrkstZ} and \ref{lemma:mfXzsrjsmG} that} 
\begin{equation}\left.
\parbox{8.3cm}{if $I$ is a boundary lamp such that $\Peak I=\Peak L$ and $\Foot I\in\set{\lcorner L,\rcorner L}$,}\,\,\right\}
\label{eq:mrRgdfZklVs}
\end{equation} 
\revised{then $L$ is the direct product of a finite chain and the two-element chain, $L'$ is a chain, and $D'$ is a \tbdia-diagram of $L'$}. 

Hence, in the rest of the proof, we assume that \eqref{eq:mrRgdfZklVs} fails. We claim that $L'$ is a slim rectangular lattice; in view of Observation \ref{obs:djbvmfsHgrG}, it suffices to show that $L'$ is not a chain. If $L'$ was a chain, then $\Peak L/\alpha=(\lcorner L\vee \rcorner L)/\alpha= (\lcorner L/\alpha)\vee (\rcorner L/\alpha)$ together with the comparability of 
$\lcorner L/\alpha$ and $\rcorner L/\alpha$ would imply that 
$\Peak L/\alpha\in\set{\lcorner L/\alpha,\rcorner L/\alpha}$, that is, at least one of $(\Peak L,\lcorner L)$ and $(\Peak L,\rcorner L)$ would belong to $\alpha=\phi(I)$. But this would contradict, say, the Swing Lemma. Therefore, $L'$ is \revised{a} slim rectangular lattice indeed.

\revised{Letting $I_2$ play the role of $I$, Figure \ref{fig-case1} shows that} if
\begin{equation}
\text{$I$ is a boundary lamp such that $\Foot I\in\set{\lcorner L,\rcorner L}$,}
\label{eq:sMkrshnDkvrMpdlC}
\end{equation} 
then we obtain $D'$ from $D$ by removing the lower left boundary chain or the lower right boundary chain and $D'$ is clearly a \tbdia-diagram of $L'$. Thus, in the rest of the proof, we assume that \eqref{eq:sMkrshnDkvrMpdlC} fails.

Next, observe that every boundary element of $D'$ with exactly one cover is a boundary element of $L$ with exactly one cover.  
The boundary edges of $D'$ are of normal slopes, and any internal (that is, non-boundary) element of $D'$ is an internal element of $L$; see Figures \ref{fig17} and \ref{fig-case1}. Thus, based on
Definition \ref{def:SRdiagram} (of \tbdia-diagrams) and  \eqref{eq:gzlSmskLdlbRdlth}, it suffices to show that for every internal lattice element $y$ of $D'$,  
\begin{equation}\left.
\parbox{7cm}{$y$ is the foot of exactly one edge in $D'$ if and only if
$y$ is the foot of exactly one edge in $L$.}
\,\,\right\}
\label{eq:lNwhddSmgRzh}
\end{equation}
But this is trivial by the definition of $D'$; indeed, if a lattice element $y$ is not omitted and $\rr$ is an edge of $L$ with $\Foot \rr=y$, then $\rr$ is permanent or changing but not vanishing. Thus, for a lattice element $y\in D'$, the number of edges with foot equal to $y$ is the same in $D'$ as in $L$, implying \eqref{eq:lNwhddSmgRzh}. This completes the first case where $\alpha\in\Con L$ is an atom.

Second, we prove the general case by induction on the \emph{height $\height\alpha$} of $\alpha$ in $\Con L$.  Since $\Con L$ is distributive, $\height \alpha$ is the length (that is, the number of edges) of any maximal chain in the interval $[0,\alpha]$ of $\Con L$. The case $\height\alpha=0$ is trivial since then  $\alpha=\set{(x,x):x\in L}\in\Con L$, $L'\cong L$, and the diagram remains unchanged. If $\height\alpha=1$, then $\alpha$ is an atom in $\Con L$, which has been settled by in first case.

So we assume that $\height\alpha>1$. Let $\beta\in \Con L$ be a lower cover of $\alpha$. Then $\height\beta=\height\alpha-1$. Let $L^\sharp$ stand for the fixed \tbdia-diagram of $L$. 
We write $L/\alpha$ in the form $L/\alpha=\set{A, A', A'',\dots}$, where $A, A', A'', \dots \subseteq L$ are the blocks of $\alpha$. Each $\alpha$-block is the union of the $\beta$-blocks included in it. Let $B_1,B_2,\dots, B_{n}$ be the $\beta$-blocks included in $A$, let $B'_1,B'_2,\dots, B'_{n'}$ be the $\beta$-blocks included in $A'$,  let $B''_1,B''_2,\dots, B''_{n''}$ be the $\beta$-blocks included in $A''$, etc.. 
For each of these blocks of $\alpha$ or $\beta$, the join of all elements of a block $X$ is denoted by the corresponding lower case letter $x$. Note the rule $x\in X$, which follows from the well-known fact that congruence blocks are convex sublattices. For example, $a:=\bigvee A\in A$, $b_1'':=\bigvee B_1''$, etc.. 
By the induction hypothesis, the subset $U:= \set{b_1, \dots, b_{n}, \, b'_1, \dots, b'_{n'},\, b''_1, \dots, b''_{n''}, \dots }$ of $L$ determines a subdiagram $U^\sharp$ of $L^\sharp$  and this subdiagram is a  \tbdia-diagram of $L/\beta$. In the diagram $U^\sharp$, the blocks of $\alpha/\beta$ are $\set{b_1,\dots, b_{n}}$, $\set{b'_1,\dots, b'_{n'}}$, $\set{b''_1,\dots, b''_{n''}}$, etc. 
Define $d:=b_1\vee  \dots \vee b_{n}$, $d':=b'_1\vee  \dots \vee b'_{n'}$, $d'':=b''_1\vee  \dots \vee b''_{n''}$, etc..
The Correspondence Theorem, see Burris and Sankappanavar \cite[Theorem 6.20]{burrissank}, and $\beta\prec \alpha$ give that $\alpha/\beta$ is an atom in $\Con(L/\beta)$. Hence the first case (dealing with atoms) and the fact that $U^\sharp$ is a \tbdia-diagram of $L/\beta$ yield that the subdiagram $V^\sharp$ of $U^\sharp$ determined by the subset  $V:=\set{d, d', d'',\dots}$ of $U$ is a \tbdia-diagram of $(L/\beta)/(\alpha/\beta)$. Hence, using that $(L/\beta)/(\alpha/\beta)\cong L/\alpha$ by the Second Isomorphism Theorem, see Burris and Sankappanavar \cite[Theorem 6.15]{burrissank}, it follows that $V^\sharp$ is a \tbdia-diagram of $L/\alpha$. 
Since the relation ``is a subdiagram of'' is transitive, $V^\sharp$ is a subdiagram of $L^\sharp$. So, to complete the proof, it suffices to show that $\set{a, a',a'',\dots}$ is the same as  $V:=\set{d, d', d'',\dots}$. But this is trivial since
\[
a=\bigvee A =\bigvee(B_1\cup \dots\cup B_{n})=\bigvee B_1 \vee\dots\vee \bigvee B_{n}=b_1\vee\dots\vee b_n=d
\]
and, similarly, $a'=d'$, $a''=d''$, etc.. The proof of Theorem \ref{thm:mrKtpnMgsRflpbsmG} is complete.
\end{proof}

\section{Inserting a multifork more generally and thrusting a lamp, a multifork, or a ladder}\label{sect:thrust}
In this section, we modify a construction and introduce another one. They could be of separate interest and \revised{we will need them} in the proof of Theorem \ref{thm:hbhdlgnGtmDzrd}. Note that $L$ is still a slim rectangular lattice and Convention \ref{conv:tbdD} applies. 
It follows from  Cz\'edli and Schmidt \cite[Lemma 15]{CzgScht-a-visual} or from Lemma \ref{lemma:nmtfrkstZ} (here) that a $4$-cell $H$ of $L$ is distributive if and only if $\gideal{\Roof H}$ contains no precipitous edge segment. This explains the terminology introduced in the following definition.

\begin{definition} We say that an interval $I$ of our slim rectangular lattice $L$ is \emph{\tEnl-distributive} if no precipitous edge segment lies in $\Enl I$.
\end{definition}

\revised{It is clear (by Lemma \ref{lemma:nmtfrkstZ}) that} if $I$ is \tEnl-distributive, then it is a chain of normal slope or a rectangular interval.
Implicitly,  the \tEnl-distributivity rather than the  distributivity of 4-cells was used in earlier papers. In particular, the definition of a multifork extension, which we recalled here between \eqref{eq:hwSukthhBf} and Lemma \ref{lemma:nmtfrkstZ}, remains valid mutatis mutandis for the case where the 4-cell $H$ is only \tEnl-distributive.  Thus, Lemmas \ref{lemma:hmKntzSth}, \ref{lemma:hBbrSzspnrJt},
 and \ref{lemma:RnGcN}, the proofs in earlier papers,
and mainly Theorem \ref{thm:mrKtpnMgsRflpbsmG} and Lemma \ref{lemma:RnlxtG} imply the following observation.

\begin{observation}\label{obs:wZrkdvrHlH} 
Lemmas \ref{lemma:nmtfrkstZ} and \ref{lemma:mfXzsrjsmG} remain valid if we replace ``distributive'' by ``\tEnl-distributive''. Furthermore, if 
$L$ is a slim rectangular lattice with a \tbdia-diagram $L^\sharp$,
$H$ is a \tEnl-distributive 4-cell of $L$,
and $L'$ and $L'^\sharp$ are obtained from $L$ and $L^\sharp$ by inserting a multifork at $H$, then $L$ is a sublattice of $L'$,  the just-inserted multifork determines a minimal lamp $I$ of $L'$,  $L\cong L'/\phi(I)$ where $\phi(I)$ is defined in \eqref{eq:sznZtcsTl}, and $L^\sharp$ is the quotient diagram  $L'^\sharp/\phi(I)$.
\end{observation}

\begin{figure}[ht] \centerline{ \includegraphics[scale=1]{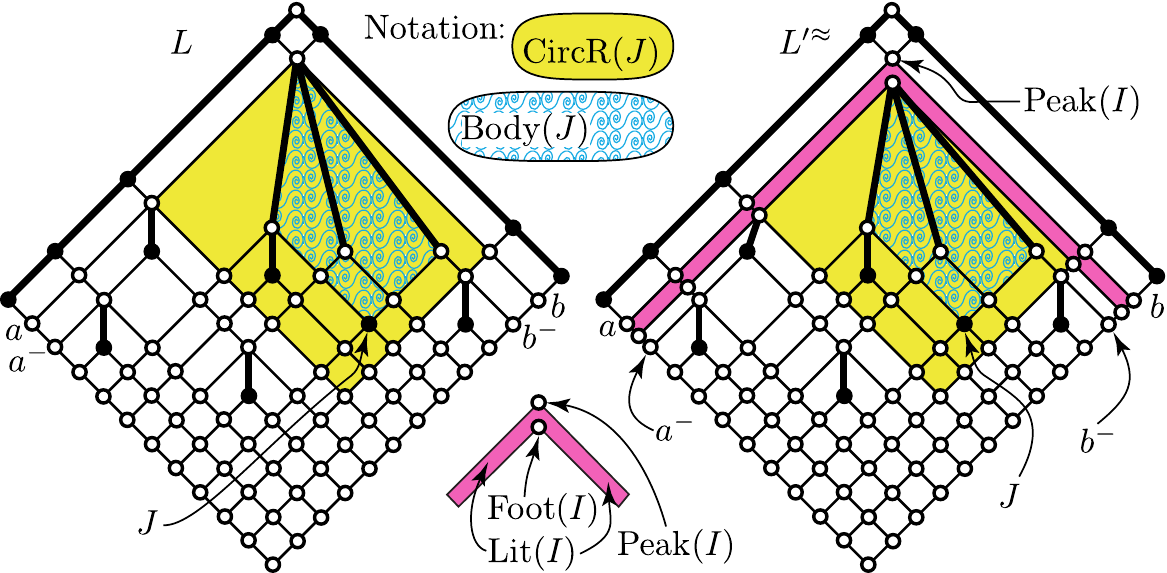}} \caption{En route to thrusting a multifork and a lamp $I$ atop an internal lamp $J$; $L'^{\approx}$ is not a lattice yet}\label{fig-18}
\end{figure}

Roughly saying, the following definition gives the converse of the quotient diagram construction by an atom $\alpha\in\Con L$ at  diagrammatic level. A \emph{ladder} is the direct product of the two-element chain and an at least two-element finite chain.

\begin{figure}[ht] \centerline{ \includegraphics[scale=1]{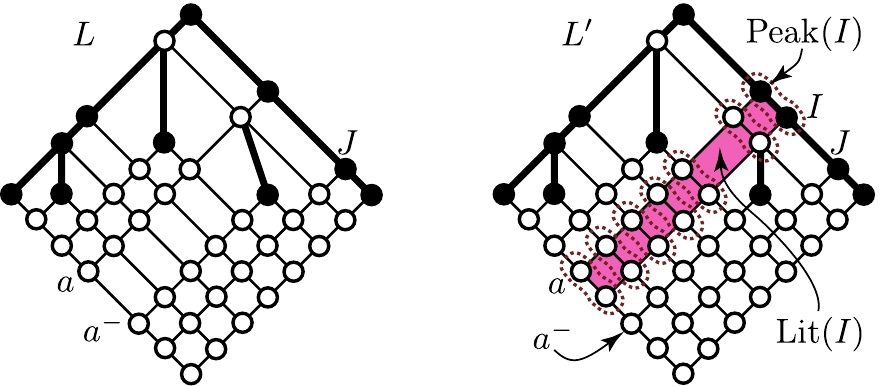}} \caption{Thrusting a lamp atop a boundary lamp $J$}\label{fig-12}
\end{figure}

\begin{definition}[Thrusting a lamp, a multifork, or a ladder]\label{def-stznkmjTvngmkZbw} 
Assume that $J$ is a lamp of a slim rectangular lattice $L$ with a fixed \tbdia-diagram $L^\sharp$. Let  $k\in\Nplu$ if $J$ is an internal lamp, and let $k:=1$ if $J$ is a boundary lamp. We \emph{thrust a $k$-fold new lamp $I$ atop $J$ $($into $L)$} to obtain a new lattice $L'$ and its \tbdia-diagram $L'^\sharp$ as follows; depending on $J$, there are two cases.

\textup{(A)} (Thrusting a lamp $I$ atop an internal lamp $J$).
First, assume also that $J$ is an internal lamp. Then 
the construction is illustrated by Figure \ref{fig-18}; by letting $(L,L',J,I):=(L',L,J_1,I)$, it is also illustrated by Figure \ref{fig17}. (So in Figure \ref{fig17}, the old lattice is $L'$ on the right while $L$ on the left is the new lattice.) 
Let $a:=\lsupp {\Peak J}$ and  $b:=\rsupp {\Peak J}$. Let  $a^-\in \LBnd L$ and $b^-\in \RBnd L$ be the (unique) lower cover of $a$ and that of $b$, respectively. 
Insert a new lattice element into the edge $[a^-,a]$, and draw a line from this new element to the northeast, with (normal) slope $(1,1)$. Similarly, insert a new lattice element into the edge $[b^-,b]$, and draw a line from this new element to the northwest, with (normal) slope $(-1,1)$. 
The intersection point of these two lines will be $\Foot I$. 
The line segments from the two new vertices up to  $\Foot I$ will be  $\LFloor I$ and $\RFloor I$. Change   $\Roof J$ to $\Floor I$. In particular, $\Peak J$ changes to $\Foot I$. The old $\Peak J$ and $\Roof J$ become $\Peak I$ and $\Roof I$, respectively. 
To every vertex $x\neq \Peak I$ lying on  $\LRoof I$
(that is, on the old $\LRoof J$ and so $x\in L$), add a new element $x^-$ lying on  $\LFloor I$ such that $x^-\prec x$ and the edge $[x^-, x]$ is of slope $(1,-1)$. (For $x=a$, this notation is consistent with the previous meaning of $a^-$.) Old edges of the form $[u,x]$ such that $u\notin \LFloor I$ become edges $[u,x^-]$. We act left-right symmetrically on \revised{$\RRoof I$ and} $\RFloor I$. 
Old edges of the form $[u,\Peak{\text{old } J}]$ with $u\notin \Roof {\text{old }J}$ change to  edges  $[u,\Foot I]$ as  $\Peak J$ changes  to $\Foot I$. We have  obtained a poset $L'^{\approx}$ and its diagram, but $L'^{\approx}$ is not a lattice. However,  we have defined $\Peak I$, $\Foot I$, $\Roof I$, and $\Floor I$, and they determine $\Enl I$ and the way how to insert a $k$-fold multifork into $\Enl I$ to turn $L'^\approx$ into a lattice $L'$ and its \tbdia-diagram $L'^\sharp$. We have defined  $L'$,  $L'^\sharp$, and $I$, as required.

\textup{(B)} (Thrusting a lamp $I$ atop a boundary lamp $J$). Let, say, $J$ be a right boundary lamp; the case of a left boundary lamp is analogous by symmetry. The construction is visualized by 
Figure \ref{fig-12}; it is also exemplified by Figure \ref{fig-case1} if we interchange $L$ and $L'$. (These two figures speak for themselves.)
Now the new lattice $L'$ and the new lamp $I\in \Lamp{L'}$ is defined in the same way as for an internal $J$ in Part 
\textup{(A)} apart from the following straightforward modifications. We let  $b':=b$, whence $\Peak I$ 
lies on $\RRoof{\textup{old }J}$ (and also on the upper right boundary of $L'$).  On the left, that is, with $\LRoof I$ and $\LFloor I$, we do the same as in  Part 
\textup{(A)}. But \revised{then we} neglect the instruction to ``act left-right symmetrically on $\RFloor I$'' and, furthermore, 
$L'$ and its diagram is the poset $L'^{\approx}$ and its diagram, respectively.

\textup{(C)} The construction described in \textup{(A)} and that described in \textup{(B)} are also called \emph{thrusting a multifork} and \emph{thrusting a ladder} atop $J$, respectively.  
\end{definition}

If we thrust a lamp $I$ atop an internal lamp $J$, then $I$ becomes an internal lamp. Similarly, if $J$ is a boundary lamp, then so is $I$. \revised{Note} that in $L'$ (which is now the ``old lattice'') on the right of Figure \ref{fig-case1}, $\Peak J=\Peak U$. We can thrust a  lamp $I$ atop $J$ to obtain $L$ (the same figure on the left), and we can also thrust $I$ atop $U$, but then we obtain an entirely different diagram and lattice.

As opposed to Observation \ref{obs:wZrkdvrHlH}, $L$ is not a sublattice of $L'$ described in Definition \ref{def-stznkmjTvngmkZbw}. However, the following lemma is motivated by Observation \ref{obs:wZrkdvrHlH}.

\begin{lemma}\label{lemma:szbzhvszPnDb}
Let $L$, $L^\sharp$, $J$, $k$, $I$, $L'$ and $L'^\sharp$ be as in Definition \ref{def-stznkmjTvngmkZbw}. Then $L\cong L'/\phi(I)$, see \eqref{eq:sznZtcsTl}, and $L^\sharp$ is the quotient diagram  $L'^\sharp/\phi(I)$. Furthermore, 
making no distinction between the old lamps and their possibly modified new variants, $\Lamp{L'}=\Lamp L\mathrel{\dot\cup}\set{I}$ (disjoint union), $\Lamp L$ is a subposet of\, $\Lamp {L'}$,  $I$ is a minimal element of $\Lamp{L'}$, and $I$ has exactly the same covers in $\Lamp {L'}$ as $J$. 
\end{lemma}

Before the proof of Lemma \ref{lemma:szbzhvszPnDb}, we need another lemma. 
Recall by \eqref{eq:ncghnvmgnBZcfDs} that the upper sides of  \revised{$\CircR W$} of an internal lamp  \revised{$W$} are edges;  
\begin{equation}\left.
\parbox{8.5cm}{let \revised{$\hhl(W)$ and $\hhr(W)$} stand for the upper left edge and the upper right edge of \revised{$\CircR W$}, respectively.}
\,\,\right\}
\label{eq:sWmslKbjjknLmt}
\end{equation} 
For a left boundary lamp $J'$, $\hhl(J')$ is the edge $J'$ itself (that is,  $\hhl(J')$ is the only neon tube of $J'$) and $\hhr(J')$ is undefined. For a right boundary lamp $J''$, $\hhl(J'')$ is undefined and  $\hhr(J'')$ is (the only neon tube of) $J''$. For $U\in\Lamp K$, 
\begin{equation}\left.
\parbox{10.5cm}{let $\Nwl U$  denote the lamp that has the neon tube of the trajectory through $\hhl(U)$, and define $\Nel U$ similarly but using $\hhr(U)$.}\,\,\right\}
\label{eq:wntcskMnctrPl}
\end{equation}
The acronyms used in \eqref{eq:wntcskMnctrPl} come from ``northwest lamp'' and ``northeast lamp''. Note that $\Nwl U=\Nel U$ can occur as well as $\Nwl U \neq \Nel U$.
\revised{For a poset $Q$, let $\Min Q$ stand for the set of minimal elements of $Q$.}

\begin{lemma}\label{lemma:cVcRtGl}\label{rem:twcvrth}
Let $I$ be an internal lamp of a  slim rectangular lattice $L$.
Then the set of covers of $I$ in $\Lamp L$ is \revised{$\Min{\set{\Nwl I,\Nel I}}$}; see \eqref{eq:wntcskMnctrPl}. 
\end{lemma}

\begin{figure}[ht] \centerline{ \includegraphics[scale=1.0]{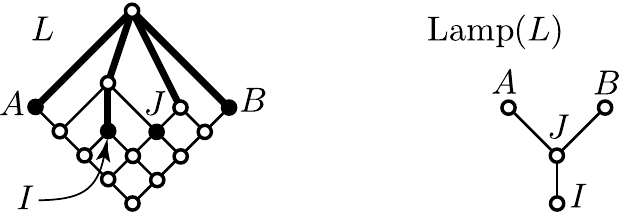}} \caption{\revised{$\Nwl I=A$, $\Nel I=J$, but $I$ has only one cover in $\Lamp L$}}\label{figK2}
\end{figure} 

\revised{
\begin{example} Figure \ref{figK2} indicates that we cannot write $\set{\Nwl I,\Nel I}$ instead of $\Min{\set{\Nwl I,\Nel I}}$ in Lemma \ref{rem:twcvrth}.
\end{example}}

\begin{remark}\label{rem:cSgslVdkRzDrg} 
The Two-cover Theorem from Gr\"atzer \cite{gG-cong-fork-ext} with Lemma \ref{lemma:vltfLm} give that $I$ has at most two covers in $\Lamp L$. Lemma \ref{lemma:cVcRtGl} implies this fact as well as the Two-cover Theorem, and says a bit more by describing the covers of $I$. 
\end{remark}

\begin{proof}[Proof of Lemma \ref{lemma:cVcRtGl}]
Let $\qq$ be a neon tube and  $\pp$ be a prime interval of $L$. Since a neon tube cannot be up-perspective to another prime interval, the Swing Lemma, see Gr\"atzer \cite{gG15} (see also Cz\'edli, Gr\"atzer, and Lakser \cite{CzGGLakser} or Cz\'edli and Makay \cite{CzGMakay} for alternative approaches) in our particular case asserts that 
\begin{equation}\left.
\parbox{10.0cm}{$\con(\qq) > \con(\pp)$ if and only if there is an $n\in\revised{\Nnul}$ and there are pairwise distinct prime intervals $\rr_0=\qq$, $\rr_1$,\dots, $\rr_{n-1}$, $\rr_n=\pp$ such that, for all $i\in\set{1,\dots,n}$,  either $\rr_{i-1}$ and $\rr_i$ are neon tubes of the same lamp,  or $\rr_{i-1}$ is down-perspective to $\rr_i$, or $\rr_{i}$ is a neon tube of a lamp \revised{$K=K(i)$} such that $\rr_{i-1}\in \set{\hhl(K),\hhr(K)}$.}\,\,\right\}
\label{eq:vnvSksztmgzrTsrgWgj}
\end{equation}
For $U\in\Lamp K$, $\con(U):=\con(\Foot U,\Peak U)$. Lemma \ref{lemma:vltfLm} and \eqref{eq:vnvSksztmgzrTsrgWgj} give that for any lamps $U_1,U_2,U_3\in \Lamp L$ and neon tubes $\intv u_i$ of $U_i$, for $i\in\set{1,2,3}$,
\begin{align}
&U_1<U_2\iff \con(\intv u_1)<\con(\intv u_2),
\label{eq:dtTmjfspRskdsNnul}\\
&\con(\intv u_3)=\phi(U_3)= \revised{\con(U_3)}\label{eq:dtTmjfspRskdsNa},\\
&\con(\hhl(U_3))=\con(\Nwl{U_3})=\phi(\Nwl {U_3}),\,\,\text{ and}
\label{eq:dtTmjfspRskdsN}\\
&\con(\hhr({U_3}))=\con(\Nel {U_3})=\phi(\Nel {U_3}).
\label{eq:dtTmjfspRskdsNc}
\end{align}

Fix a neon tube $\pp$  of $I$.
First, assume that $J\in \Lamp L$ such that $I<J$. Let $\qq$ stand for a neon tube of $J$.  Since $\con(\qq)>\con(\pp)$ by  
\eqref{eq:dtTmjfspRskdsNnul},  there is a sequence described in \eqref{eq:vnvSksztmgzrTsrgWgj} from $\qq$ to $\pp$. 
For any neon tube $\pp'$ of $I$, no edge can be down-perspective to $\pp'$ since  $\Foot{\pp'}\in\Mir L$. In particular,  $\rr_{n-1}$ cannot be down-perspective to $\rr_n=\pp$. Hence, letting $j$ denote the largest subscript such that $\rr_j$ is not a neon tube of $I$, it follows that $j<n$ and  $\rr_j\in\set{\hhl(I),\hhr(I)}$. Cutting the sequence into two parts at $\rr_j$ and applying \eqref{eq:vnvSksztmgzrTsrgWgj} to both parts, we obtain that $\con(\qq)\geq \con(\rr_j)>\con(\pp)$. This fact, $\rr_j\in\set{\hhl(I),\hhr(I)}$, and 
\eqref{eq:dtTmjfspRskdsNnul}--\eqref{eq:dtTmjfspRskdsNc} yield that if $I<J$, then $\phi(I)<\phi(\Nwl I)\leq \phi(J)$ or $\phi(I)<\phi(\Nel I)\leq \phi(J)$. But $\phi$ is an order isomorphism, whereby if $I<J$, then $I<\Nwl I\leq J$ or 
$I<\Nel I\leq J$. This implies Lemma \ref{lemma:cVcRtGl} since the inequalities  $I<\Nwl I$ and  $I<\Nel I$ clearly hold by \eqref{eq:vnvSksztmgzrTsrgWgj} and \eqref{eq:dtTmjfspRskdsNnul}.
\end{proof}

\begin{proof}[Proof of Lemma \ref{lemma:szbzhvszPnDb}]
Except for its part about the covers, the  lemma  follows in a straightforward way from definitions and \eqref{eq:mnHhwlCh}; these tedious details will be omitted. In a detailed way, we only deal with the covers of  $I$ and $J$.
First, assume that $J$ is a boundary lamp. So is $I$ \revised{by the paragraph following Definition \ref{def-stznkmjTvngmkZbw}(C)}. By \cite[Lemma 3.2]{CzGlamps} (or by Lemma \ref{lemma:vltfLm} here), the boundary lamps are exactly the maximal elements of $\Lamp L$. Hence, $I$ and $J$ have the same set (the empty set) of covers. 

Second, let $J$ be an internal lamp. By construction, so is $I$. 
It follows from the construction of $L'$ that, in $L'$, 
$\hhl(I)$ and $\hhl(J)$ belong to the same trajectory; see Figure \ref{fig17} (where $L$ and $L'$ are interchanged and $J=J_1$). Therefore, $\Nwl I=\Nwl J$. Similarly, $\Nel I=\Nel J$. Thus, Lemma \ref{lemma:cVcRtGl} implies that $I$ and $J$ have the same covers in $\Lamp{L'}$, completing the proof of Lemma \ref{lemma:szbzhvszPnDb}.
\end{proof}

\section{Proving Theorems \ref{thm:stnKzrzmG} and \ref{thm:hbhdlgnGtmDzrd}}\label{sect:tmprfs} 

\begin{proof}[Proof of Theorem  \ref{thm:hbhdlgnGtmDzrd}]
In view of \revised{Remark \ref{rem:cskGgJzSfHm} and} Lemma \ref{lemma:vltfLm}, we can assume that 
$P\cong\Lamp L$ for  a slim rectangular lattice $L$ and $J\in\Lamp L$ corresponds to $j\in P$.  Convention \ref{conv:tbdD} applies. Thrust a lamp $I$ atop $J$ to obtain $L'$; see Definition \ref{def-stznkmjTvngmkZbw}. It follows from Definition \ref{def:kbClbtsmThkm} and Lemmas \ref{lemma:vltfLm} and \ref{lemma:szbzhvszPnDb} that
\[\Jir{\Con{L'}}\cong \Lamp{L'}\cong \brosum{\Lamp L} J I\cong 
\brosum P j i, \text{ completing the proof.}\qedhere
\]
\end{proof}

\begin{figure}[ht] \centerline{ \includegraphics[scale=1.0]{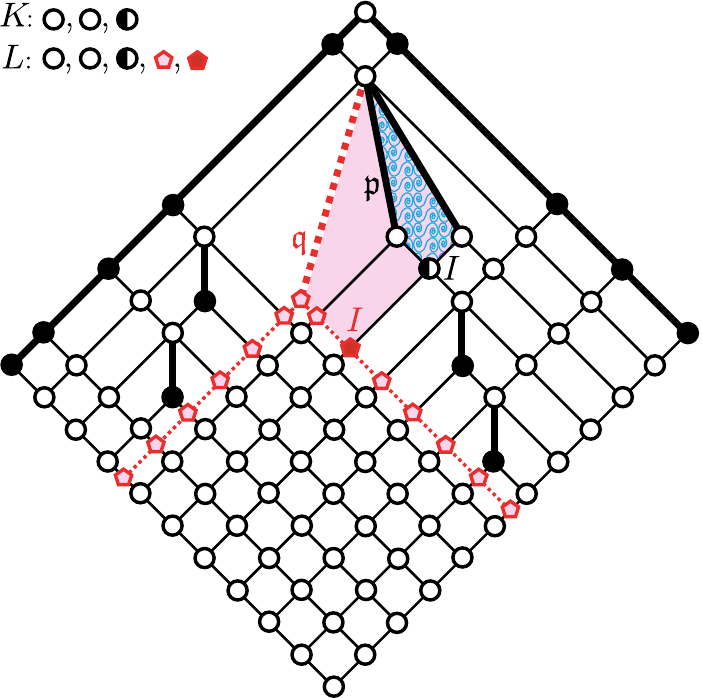}} \caption{Illustrating the proof of  Lemma \ref{lemma:mgntBlht}}\label{fig-int}
\end{figure}

\begin{lemma}\label{lemma:mgntBlht} If $n\in\Nplu$ and $K$ is slim rectangular lattice, then there is a slim rectangular lattice $L$ such that each internal lamp of $L$ (or just one specified internal lamp of $L$) has at least $n$ neon tubes and \ $\Lamp L\cong \Lamp K$.
\end{lemma}

\begin{proof} It suffices to increase the number of the neon tubes of an internal lamp $I\in\Lamp K$ by 1; see Figure \ref{fig-int}, where $K$ consists of the \revised{black and solid} edges and the circle-shaped vertices drawn in black. In this figure, $\Body I$ is \spiralfil-filled. 
(As opposed to the general convention, to indicate that $\Foot I$ is going to change, $\Foot I$ is only half-filled in the figure.) Let $\pp$ be the leftmost neon tube of $I$.
We obtain $L$ and its diagram by adding a new leftmost lamp $\qq$ to $I$  (by inserting a fork) such that $\Foot \qq$ is (geometrically) very close to $\Foot \qq\wedge \Foot\pp$. The new elements are \revised{red and pentagon-shaped},   the new edges are drawn in red dotted lines, and the body of the new $I$ is pink-filled. By  (the last sentence of) \cite[Theorem 11]{CzgScht-a-visual}, $L'$ is a slim rectangular lattice. It is straightforward to derive from the construction of $L$ and Lemma \ref{lemma:vltfLm} that $\Lamp L\cong\Lamp K$,  proving Lemma \ref{lemma:mgntBlht}.
\end{proof}

\begin{figure}[ht] \centerline{ \includegraphics[scale=1.0]{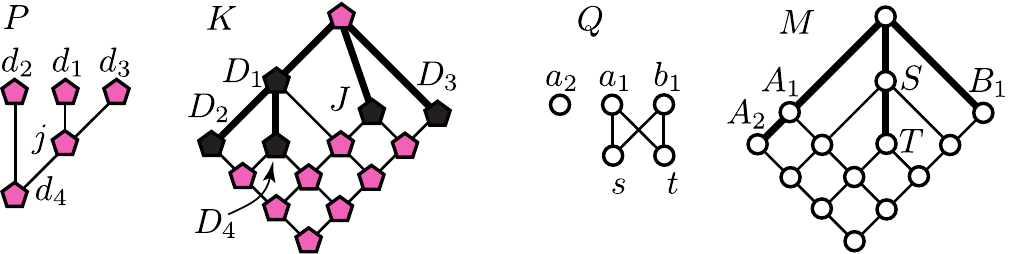}} \caption{Illustrating the proof of \revised{Theorem  \ref{thm:stnKzrzmG}}}\label{fig9}
\end{figure}

\begin{proof}[Proof of Theorem \ref{thm:stnKzrzmG}]
Assume that $P\cong \Lamp K$ and $Q\cong\Lamp M$ for some slim rectangular lattices $K$ and $M$; see Figure \ref{fig9} for an example. Let $j$ be a non-maximal element of $P$, and let $J$ be the corresponding  lamp of $K$. As a rule, 
a lamp corresponding to an element $x$ of $P$ or $Q$ is denoted by the corresponding capital letter $X$.
By \cite[Lemma 3.2]{CzGlamps}, $J$ is an internal lamp. 
Assume that 
\begin{equation}\left.
\parbox{11cm}{the (fixed) multifork sequence for $M$ (in the sense of Lemma \ref{lemma:nmtfrkstZ}) begins with an $f$-by-$g$ grid $M_0$ (which consists of $(f+1)(g+1)$ elements).}\,\,\right\}
\label{eq_pNgfrkTsbgzpRcTs}
\end{equation}
In Figure \ref{fig9}, $f=2$ and $g=1$. 
Motivated by Lemma \ref{lemma:hBbrSzspnrJt}, 
\begin{equation}\left.
\parbox{9.4cm}{a neon tube $\pp$ of $L$ is \emph{used} if
a precipitous edge segment lies in  
 $\LHOT\pp$; in the opposite case, we say that $\pp$ is \emph{unused}.}\,\,\right\}
\label{eq:sZsrjhlfjDtmfNkv}
\end{equation}
A precipitous edge segment comes from a neon tube of an internal lamp $U$. If such an edge segment lies in $\LHOT\pp$, then we say that the internal \emph{lamp $U$ uses $\pp$}.
We claim that 
\begin{equation}
\text{no internal lamp uses three different neon tubes of $J$.}
\label{eq:mnvrBsfnKfSrj}
\end{equation}
For the sake of contradiction, suppose that $\pp_1$, $\pp_2$, and $\pp_3$ are pairwise distinct neon tubes of $J$ and they \revised{all are used by an} internal lamp $U$. It follows from  Lemma \ref{lemma:hBbrSzspnrJt} that, for $i\in\set{1,2,3}$, $\Body U$ lies in $\LHOT{\pp_i}$. Let $m$ be the geometric midpoint of 
a neon tube  $\rr$ of $U$. 
\revised{For $i\in\set{1,2,3}$,
$\LHOT{\pp_i}$ is the union of its left part and its right part; both of these parts are normally bordered territories by \eqref{eq:znksnkhLpktTj}. Hence, using that $\rr$ is precipitous, we obtain that $m$ belongs to the geometric interior of one of these two parts. By left-right symmetry, we can assume that for two values of $i\in\set{1,2,3}$, $m$ belongs to the interior of the \emph{left} part of  $\LHOT{\pp_i}$. 
As each of Figures \ref{fig3}, \ref{fig-16}, \ref{fig17}, \ref{fig-18}, \ref{fig-int}, and \ref{figj12} clearly shows, these two left parts have disjoint interiors. Thus, we have obtained a contradiction that proves}  
\eqref{eq:mnvrBsfnKfSrj}.

As a consequence of  \eqref{eq:mnvrBsfnKfSrj}, we obtain   that the number of used neon tubes of $J$ is at most $2\cdot|\Lamp K|$.
Apply  Lemma \ref{lemma:mgntBlht}  to modify $K$ so that (the isomorphism class of) $\Lamp K$ does not change but
$J$ has at least $4\cdot (1+2 \cdot|\Lamp K|)$ neon tubes. \revised{Let us} list these neon tubes from left to right, according the their positions in the fixed \tbdia-diagram of $K$. We can partition this list \revised{into} $(1+2 \cdot|\Lamp K|)$ many four-element sublists such that each sublist consists of  four consecutive neon tubes of $J$. One of these \revised{sublists} consists of unused neon tubes since otherwise $K$ would have more than 
$2\cdot|\Lamp K|$  used neon tubes. 
Therefore, we can take four consecutive unused neon tubes, $\qq_1$, \dots, $\qq_4$, of $J$; see Figure \ref{figj12}, where (the new version of) $K$ only consists of the pentagon-shaped elements. 

Let $A_1$,\dots,$A_f$  be the left boundary lamps of $M$, listed downwards. Similarly, let $B_1$,\dots,$B_g$  be the right boundary lamps of $M$, listed downwards again; see  Figure \ref{fig9}, where, remember, $f=2$ and $g=1$.

\begin{figure}[ht] \centerline{ \includegraphics[scale=1.0]{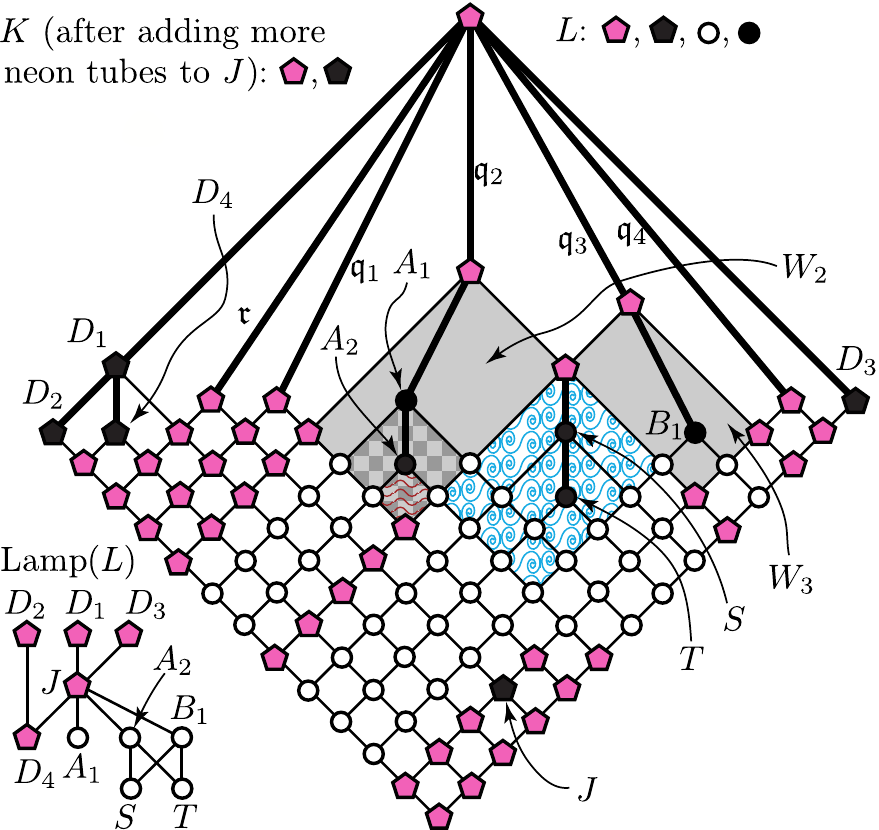}} \caption{The construction for the example given in Figure \ref{fig9}}
\label{figj12}
\end{figure}

For $i\in\set{2,3}$, let $W_i$ stand for the unique 4-cell of $K$ such that $\Peak {W_i}=\Foot{\qq_i}$. In Figure \ref{figj12}, these two 4-cells (as geometric shapes) are grey-filled (but they are not 4-cells in the 
\revised{final sate, $L$}).
Since  $\qq_1$, \dots, $\qq_4$ are unused, $W_2$ and $W_3$ are \tEnl-distributive. Add a lamp with a single neon tube to $W_2$ 
and label this lamp by $A_1$. That is, we insert a fork (that is, a 1-fold multifork) into $W_2$, and $A_1$ is the lamp we obtain  in this way. 
Then there is unique 4-cell $W_2'$ with $\Peak{W_2'}=\Foot{A_1}$. To this 4-cell, which is ``chess-board-filled'' in Figure \ref{figj12}, add a lamp labelled by $A_2$ that has only one neon tube. Continue with the rest of  $A_1,\dots, A_f$ until all of them become lamps ``piled up'' in the rectangular interval (originally a 4-cell) $W_2$. (In Figure \ref{figj12}, $f$ is only $2$. If $f$ was greater than $2$, then we would add $A_3$ to the ``chess-board-and-wave''-filled 4-cell the peak of which is $\Foot {A_2}$.) 
We similarly add lamps $B_1$, \dots, $B_g$ that are piled up in (the geometric area of) $W_3$. Now the $\REnl{A_i}$'s, $i=1,\dots,f$, and the 
 $\LEnl{B_i}$'s, $i=1,\dots,g$, determine an $f$-by-$g$-grid, which is indicated by the \spiralfil-filled normally bordered rectangle in Figure \ref{figj12}. 
Letting this rectangle play the role of $M_0$, we can perform the same multifork extensions as in \eqref{eq_pNgfrkTsbgzpRcTs} to turn the above-mentioned grid to an interval isomorphic to $M$. 
Let $L$ denote the lattice (diagram) we obtain in this way; see Figure \ref{figj12}.
Since the neon tubes $\qq_1$, \dots, $\qq_4$ were unused 
before we began to add the lamps $A_1$,\dots, $A_f$, $B_1$, \dots, $B_g$,  Lemmas \ref{lemma:vltfLm} and \ref{lemma:cVcRtGl} yield in a tedious but straightforward  way that $\Jir{\Con L}\cong \Lamp L\cong Q\jsum j P$, as required. 
\end{proof}

\semmi{
\section*{Statements and declarations}
\subsection*{Data availability statement} Data sharing is not applicable to this article as no datasets were generated or analyzed during the current study.
\subsection*{Competing interests}  Not applicable as there are no interests to report.
}

\end{document}